\documentclass{article}
\usepackage[sort]{cite}
\usepackage[utf8]{inputenc} 
\usepackage{microtype}

\usepackage{PRIMEarxiv}
\usepackage{graphicx} %

\usepackage{nicefrac}
\usepackage{mathtools}
\usepackage{mathrsfs}
\usepackage{amsmath, amssymb, amsopn, amsthm} 
\usepackage{lipsum}
\usepackage{amsfonts}
\usepackage{graphicx}
\usepackage{epstopdf}
\usepackage{fancyhdr}       %

\usepackage{tikz}
  \usetikzlibrary{arrows,automata}
  \usetikzlibrary{spy}

\usepackage{booktabs} %
\usepackage{array} %
\usepackage{paralist} %

\usepackage{verbatim} %
\usepackage{subcaption} %
\usepackage{float}
\usepackage{mathtools}
\usepackage{mathrsfs}
\usepackage{multirow}
\usepackage{mdwlist}
\usepackage{multicol}
\usepackage[colorlinks]{hyperref}
\definecolor{GGreen}{RGB}{0,128,0}
\hypersetup{
    linkcolor= {GGreen},%
    citecolor={blue}, %
    urlcolor={cyan}
}

\usepackage[capitalize,noabbrev,nameinlink]{cleveref}
\crefformat{equation}{(#2#1#3)}

\usepackage[textsize=tiny]{todonotes}

\title{Weakly Convex Regularisers for Inverse Problems:\\Convergence of Critical Points and Primal-Dual Optimisation}

\author{
  Zakhar Shumaylov\\
  University of Cambridge \\
  \texttt{zs334@cam.ac.uk} \\
   \And
  Jeremy Budd \\
  California Institute of Technology \\
  \texttt{jmbudd@caltech.edu} \\
  \AND
 Subhadip Mukherjee \\
  IIT Kharagpur \\
\texttt{subhadipju@gmail.com} \\
\And
  Carola-Bibiane Sch\"onlieb \\
  University of Cambridge \\
  \texttt{cbs31@cam.ac.uk} \\
 \
}

\numberwithin{equation}{section}
\newtheoremstyle{exampstyle}
{4pt} %
{4pt} %
{\itshape} %
{} %
{\bfseries} %
{.} %
{.5em} %
{} %
\theoremstyle{exampstyle}

\DeclareMathOperator*{\argmin}{argmin}

\newcommand{\be}{\begin{equation}}
\newcommand{\ee}{\end{equation}}
\newcommand{\bes}{\begin{equation*}}
\newcommand{\ees}{\end{equation*}}

\newtheorem{theorem}{Theorem}[section]
\newtheorem{proposition}[theorem]{Proposition}
\newtheorem{lemma}[theorem]{Lemma}

\theoremstyle{definition}
\newtheorem{definition}[theorem]{Definition}
\newtheorem{assumption}[theorem]{Assumption}
\newtheorem{nb}[theorem]{Note}
\newtheorem{example}[theorem]{Example}

\theoremstyle{remark}

\newcommand{\Op}[1]{\operatorname{\mathcal{#1}}}
\newcommand{\scv}[1]{#1}
\newcommand{\wcv}[1]{#1}

\newcommand{\A}{\Op{A}}
\newcommand{\X}{\Op{X}}
\newcommand{\Y}{\Op{Y}}
\newcommand{\R}{\Op{R}}

\newcommand{\D}{\Op{D}}
\newcommand{\x}{x}

\crefname{assumption}{assumption}{assumptions}

\begin{document}
	\maketitle

\begin{abstract}

Variational regularisation is the primary method for solving inverse problems, and recently there has been considerable work leveraging deeply learned regularisation for enhanced performance. However, few results exist addressing the convergence of such regularisation, particularly within the context of critical points as opposed to global minimisers. In this paper, we present a generalised formulation of convergent regularisation in terms of critical points, and show that this is achieved by a class of weakly convex regularisers. We prove convergence of the primal-dual hybrid gradient method for the associated variational problem, and, given a Kurdyka--\L ojasiewicz condition, an $\mathcal{O}(\log{k}/k)$ ergodic convergence rate. 
Finally, applying this theory to learned regularisation, we prove universal approximation for input weakly convex neural networks (IWCNN), and show empirically that IWCNNs can lead to improved performance of learned adversarial regularisers for computed tomography (CT) reconstruction.  \\

\textbf{Key words:} Image reconstruction, inverse problems, learned regularisation, weak convexity, deep learning, computed tomography, convergent regularisation, primal-dual optimisation.
\end{abstract}
\begin{figure}
     \centering
     \includegraphics[width=\linewidth,clip,trim ={0.2 0.2 0.2 0.2}]{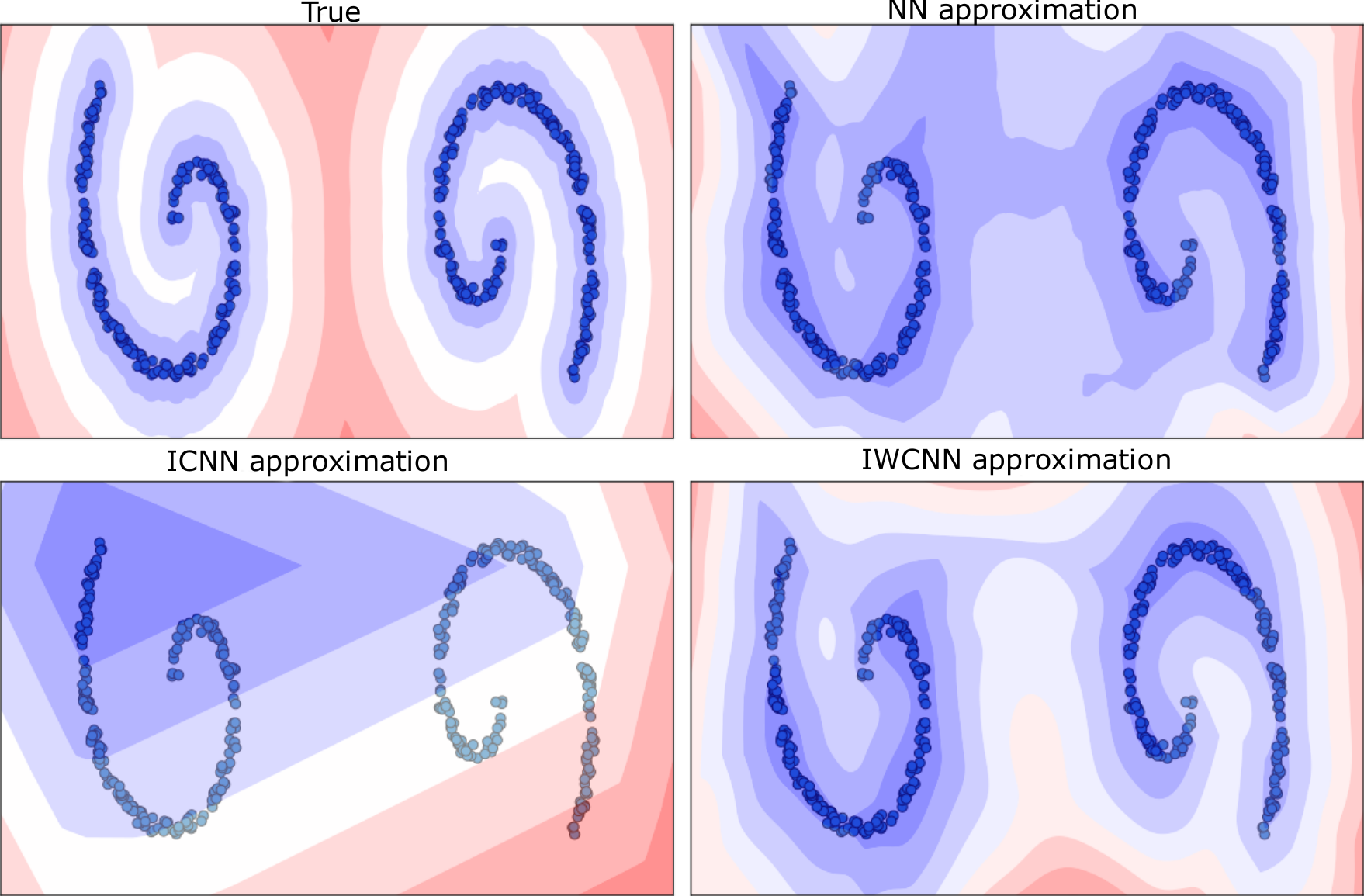}
     \caption{\footnotesize{Contour plots comparing the distance function to a data manifold (top-left) with learned regularisers for denoising this data, via convex (bottom-left), standard (top-right), and weakly convex (bottom-right) adversarial regularisation. The convex regulariser lacks an interpretation as a distance function, whilst the weakly convex regulariser, introduced in this work, retains this feature and shows improved generalisation.}}
     \label{fig:toy_example}
 \end{figure}

\section{Introduction}\label{sec:intro}

In an inverse problem, one seeks to recover an unknown parameter (e.g., an image) that has undergone some, typically lossy and noisy, transformation (e.g., during measurement). Such problems arise frequently in science, e.g. in medical imaging such as in MRI, CT, and PET, and also beyond science, e.g. in art restoration \cite{Natterer2001,scherzer2009variational,calatroni_d’autume_hocking_panayotova_parisotto_ricciardi_schönlieb_2018}. 

Mathematically, one seeks to estimate an unknown parameter $x^* \in \X$ from a transformed and noisy measurement
\begin{equation}
    y^{\delta} = \mathcal{A}x^*+ \eta \in\Y,
    \label{inv_prob_data}
\end{equation}
see \cite{engl1996regularization}. Here $\X$ and $\Y$ are Banach spaces; $\mathcal{A}:\X\rightarrow \Y$ is the \emph{forward model}, assumed to be linear and bounded, e.g. representing imaging physics; and $\eta \in \Y$, with $\left\|\eta \right\|_{\Y}\leqslant \delta$, describes measurement noise. However, \cref{inv_prob_data} is usually \emph{ill-posed}, i.e. $x \in \X$ with $\A x = y^\delta $ may not exist, be unique, or be continuous in $y^{\delta}$.

An influential technique to overcome this ill-posedness has been the variational approach pioneered by \cite{Tikh1963} and \cite{phillips1962}:
\begin{equation}
\label{eq:ful_invprob}
\hspace*{-0.66em}x^\alpha(y^\delta) \in \argmin _{x \in \X} \left\{\alpha \mathcal{R}(x) + \mathcal{D}\left(\A x,y^\delta\right) =: \mathcal{J}_{\alpha, y^\delta}(x)\right\}.
\end{equation}

Here: $\mathcal{R}:\X \rightarrow{[0,\infty)}$ is the \emph{regulariser}, which aims to incorporate prior knowledge about $x^*$ and penalise $x$ if it is not `realistic'; $\mathcal{D}:\Y \times \Y \rightarrow{[0,\infty)}$ is the \textit{data-fidelity term}, which quantifies the distance between $y^\delta$ and $ \A x$ and comes from the distribution of the noise $\eta$; and $\alpha>0$ balances the relative importance of the two terms. A regularisation method is \emph{well-defined} if $x^\alpha(y^\delta)$ is stable with respect to data perturbations and converges, as $\delta\downarrow 0$, to a solution of the noiseless operator equation $\A x = y^0:=\A x^*$.

A good choice of regulariser $\mathcal{R}$ is important for achieving accurate results. Traditional regularisation was \textit{knowledge-driven}: the regulariser was a hand-crafted functional designed to encourage the reconstruction $x$ to have structures known to be realistic. Over the past decades, a whole zoo of such regularisers have been proposed, see e.g. \cite{ROF,BKPock2011, GilboaOsher2009, mallat1999wavelet,kutyniok2012multiscale,Besovnorms,PDEconstrained} and see \cite{Benning_Burger_2018} for a detailed overview. 
However, there is a limit to the complexity of patterns that can be described by a hand-crafted functional. %

Seeking to overcome this limitation, there has been considerable recent interest in \textit{data-driven} approaches to inverse problems, see \cite{Arridge2019} for a detailed overview. A key data-driven technique is to learn the regulariser from data. Such methods include: dictionary learning (see \cite{ChenNeedall16} for an overview), plug-and-play priors \cite{PnP,PnP2023}, regularisation by denoising \cite{RED}, deep image priors \cite{DIP,dittmer2019regularization}, network Tikhonov (NETT) \cite{li2020nett}, total deep variation \cite{TDV2020}, 
generative regularisation (see \cite{dimakis_2022} for an overview), and adversarial regularisation \cite{AR}. These methods typically outperform knowledge-driven regularisation in practice. 

One of the major advantages of the learned regularisation paradigm compared to \textit{end-to-end} learning methods for solving inverse problems, e.g. \cite{postprocessing_cnn,Zhu2018,lpd_tmi}, is that it can operate in the \textit{weakly supervised} setting, i.e. when one has data of ground truths (and measurements) but not of ground-truth-measurement pairs. This setting is often more realistic, especially in medical imaging where paired data can be hard to acquire.
Furthermore, because they do not required paired training data, learned regularisation methods can adapt to other forward models without requiring retraining (see e.g. \cite{Yu2019,lunz_2021}).

\subsection{Motivation and contributions}\label{sec:intro11}

A key challenge in the data-driven paradigm is proving guarantees about the behaviour of the learned models, both for the inverse problem and for optimisation methods for solving it. In \cite{mukherjee2023learned} the role of convexity, via the input convex neural network architecture \cite{ICNN}, was emphasised as a method for achieving such guarantees, e.g. in the case of the adversarial convex regulariser (ACR) \cite{ACR}. %
But convexity is quite a restrictive condition, and the ACR sacrifices the adversarial regulariser's interpretability as a distance function to the data manifold (see \Cref{fig:toy_example}), contributing to worse performance. It has been observed in the literature that non-convex regularisers often have better performance  \cite{4663911,roth2009fields,pieper2022nonconvex}.
Of particular interest is recent work on `optimal' regularisation \cite{alberti2021learning, leong2023optimal}, as such optimal regularisers are typically non-convex. This non-convexity returns us to a setting in which provable guarantees are challenging to achieve. However, the optimal regulariser of \cite{alberti2021learning} is often weakly convex, and weak convexity offers a more structured setting for reasoning about the non-convex case, see e.g.  \cite{pinilla2022improved}.

Optimisation methods in the context of weak convexity can be split into either subgradient or proximal based methods; see \cite{drusvyatskiy2020subgradient} for an introductory review. Subgradient methods admit strong iteration complexity guarantees \cite{davis2018subgradient,davis2018stochastic,goujon2023learning}. 
However, the convergence settings for these methods are somewhat narrow. 
Proximal methods %
admit significantly improved complexity bounds, as well as objective convergence rates \cite{hurault2022proximal}. These rates can be further improved through prox-linearisation \cite{drusvyatskiy2019efficiency,davis2022proximal} or adaptive steps \cite{bohm2021variable}. %

Of particular interest for optimising %
\Cref{eq:ful_invprob} is the proximal primal-dual method, first analysed for convex functions in \cite{ChambollePock}. The first %
guarantees for primal-dual %
in the context of weak convexity were shown in \cite{mollenhoff}, which assumed the \Cref{eq:ful_invprob} to be convex, with a unique minimiser. In the non-convex case, however, rather stringent assumptions had to be made. This was overcome in the Arrow--Hurwicz \cite{Arrow} special case %
in \cite{SUN201815}, via a Kurdyka--{\L}ojasiewicz assumption. In \cite{guo2023preconditioned} a preconditioned primal-dual method was proposed for smooth primal functions, and convergence achieved under operator surjectivity. But for general weakly convex functions, no such convergence was shown. %
We are therefore motivated to fill this gap.

Inspired by the optimisation literature, there has been recent interest in learning weakly convex regularisers \cite{hurault2022proximal,goujon2023learning,shumaylov2023provably}.  
In \cite{shumaylov2023provably}, an input weakly convex neural network (IWCNN) architecture was introduced, and incorporated into the adversarial regularisation framework to learn a convex-nonconvex regulariser (see \cite{lanza2022convex}) with provable guarantees. However, existing convergent regularisation guarantees focus on the convergence \textbf{of} global minimisers.  In practice, optimisation schemes for \eqref{eq:ful_invprob} do not converge \textbf{to} global minimisers, but only \textbf{to} critical points, and as such one should consider convergence \textbf{of} critical points.

There exist only a few works focused on stability and convergence of regularisation in the sense of critical points. 
Under injectivity of $\A$,  \cite{durand2006stability} prove stability guarantees in the finite dimensional setting. In \cite{obmann2022convergence} stability and convergence was shown for so-called $\phi$-critical points, defined in terms of $\phi$-subgradients. But these $\phi$-critical points are \emph{not} true critical points of $\mathcal{J}_{\alpha, y^\delta}$. They can be interpreted as bounded points, with bound quality depending critically on choice of $\phi$,
motivating a return to the usual notions of subgradients. For exact definitions and further discussion, the interested reader is referred to \cite{obmann2022convergence,OBMANN2024128605}. %
The setting of critical points of $\mathcal{J}_{\alpha, y^\delta}$  is a special case of %
\cite{obmann2023convergence}, which showed stability and convergence assuming weak-to-weak continuity and appropriate coercivity of the gradient of the regulariser. But in the context of learned regularisation, differentiability has empirically been observed to worsen performance. 

In this study, we show that imposing weak convexity on the regulariser achieves the best of both worlds. It yields guarantees for both inverse problems and optimisation, whilst also leading to robust numerical performance. We illustrate this from three sides. On the inverse problem front we: \begin{itemize}
    \item[i.] Prove guarantees about existence, in \Cref{thm:existance}, and stability, in \Cref{thm:stability}, of solutions to \eqref{eq:ful_invprob}.
    \item[ii.] Formulate convergent regularisation in terms of critical points in a generalised way in \Cref{def:rminsol}, and prove that it is achieved by a class of weakly convex regularisers in \Cref{thm:convreg}.
\end{itemize}
On the optimisation front we: 
\begin{itemize}
    \item[iii.] Prove convergence guarantees of primal-dual methods for solving \eqref{eq:ful_invprob} under weak convexity in \Cref{thm:pd_results}.
    \item[iv.] Prove an $\mathcal{O}(\log{k}/k)$ ergodic convergence rate for the primal dual scheme under K{\L} in \Cref{thm:ergodic_rate}.
\end{itemize}
On the learned regularisation front we:
\begin{itemize}
    \item[v.] Prove a universal approximation theorem for IWCNNs in \Cref{thm:universal_IWCNN}.
    \item[vi.] Define an adversarial weak convex regulariser (AWCR) in \Cref{def:AWCR} and corroborate the theoretical results via numerical experiments with AWCRs for CT reconstruction, shown in \Cref{sample-table} and \Cref{fig:visual}.
\end{itemize}

\section{Groundwork}

Throughout this paper, $\X$ and $\Y$ denote Banach spaces.

\begin{definition}[$\rho$-convexity]
Let $f: \X \to\mathbb{R}$. %
Then $f$ is said to be $\rho$-convex if there exists some $\rho\in \mathbb{R}$ such that $f - \frac12 \rho \|\cdot\|_{\X}^2$ is convex.
 Then $f$ is said to be $\rho$-strongly convex if $\rho>0$, convex if $\rho = 0$, 
 and ($-\rho$)-weakly convex if $\rho < 0$. Hence, strongly convex entails convex, which entails weakly convex. 
\end{definition}
\begin{definition}[Subdifferential; see e.g. \cite{Kruger2003}]
    Let $f:\X \to \mathbb{R}  \cup \{\infty\}$. Then the \emph{(Frech\'et) subdifferential} of $f$ at $x$ is the empty set if $f(x)=\infty$, and if $f(x)<\infty$ is %
    \begin{align*}
    \partial f(x) := \{ \psi \in \X^* : 
    f(x') \geqslant f(x) + \psi(x'-x) + o(\|x'-x\|_{\X}) 
    \text{ as } x' \to x  
    \}, 
    \end{align*}
    where $\X^*$ denotes the dual space of $\X$, i.e. the space of all continuous linear functionals on $\X$. %
\end{definition}
\begin{nb}
    If $x$ is a local minimiser of $f$, then $0 \in \partial f(x)$.
    If $f$ is (Frech\'et) differentiable at $x$, then $\partial f(x) = \{ \nabla f(x) \}$.
\end{nb}
In what follows we are going to be interested in critical points of the objective function $\mathcal{J}_{\alpha, y^\delta}$. For this we define the set of critical points of $f:\X \to \mathbb{R} \cup \{\infty\}$ as:
\[
\operatorname*{crit}_{x \in \X} f :=\left \{ x \in \X : 0 \in \partial f\left(x\right) \right\} .
\]
For further details on convex analysis, see \Cref{ap:convanalysis}.

\subsection{Convergent regularisation}
The idea behind convergent regularisation is that we do not want to over-regularise or regularise in the wrong way. We want to guarantee that as we tune the noise level $\delta$ down, we can also tune down the regularisation parameter $\alpha$ at an appropriate rate such that all limit points of the resulting reconstructions $x^\alpha(y^\delta)$ are regular and data-consistent. To make this precise, we make the following definitions. 

\begin{definition}[$\R$-minimising and $\R$-criticising solutions]\label{def:rminsol}
    Let $y^0:= \A x^*$ be the clean measurement. If
    \begin{equation*}
    x^{\dagger} \in \argmin_{x\in\X}\R(x) \text{\,\,subject to\,\,}\Op{A}x=y^{0},
\end{equation*}
 then following \cite{poschl2008tikhonov} we call $x^{\dagger}$ an \emph{$\R$-minimising solution}, and considering critical points if
 \begin{align*}
    x^{\dagger} &\in \operatorname*{crit}_{x\in \X}\R(x)  + \iota_{\{0\}}(\A x - y^0) %
\end{align*}
 then we call $x^{\dagger}$ an \emph{$\R$-criticising solution}, 
 where the indicator function
 \[
 \iota_{\{0\}}(z) := \begin{cases}
     0, &\text{if } z=0,\\
     \infty, &\text{otherwise,}
 \end{cases}
 \]
 imposes the condition $\A x = y^0$. In particular, if $x^{\dagger}$ is $\R$-minimising, then it is also $\R$-criticising.
 
\end{definition}
Traditionally in inverse problems, convergence of regularisation has been studied in the sense of conditions under which, as $\delta \to 0$, there exist $\alpha(\delta) \to 0$ such that the limit points of the \textit{global minimisers} of $\mathcal{J}_{\alpha(\delta), y^\delta}$ are $\R$-minimising solutions. I.e., convergence \textbf{of} the minimisers. But in practice optimisation methods rarely converge \textbf{to} global minimisers when $\mathcal{J}_{\alpha, y^\delta}$ is non-convex, converging instead \textbf{to} other critical points. Thus, in this work we will instead ask whether the limit points \textbf{of} the \textit{critical points} of $\mathcal{J}_{\alpha(\delta), y^\delta}$ are $\R$-criticising solutions.

Well-defined regularisation via global minimisers is typically shown by assuming the following pre-compactness condition \cite{grasmair2010generalized,poschl2008tikhonov}: 
For all $\alpha>0$, $y \in \Y$, and $t \in \mathbb{R}$, the set $\left\{x \in \X: \mathcal{J}_{\alpha,y}(x) \leqslant t\right\}$ is sequentially pre-compact. 
In particular, this is satisfied by coercive $\R$, though coercivity may not be necessary, 
see \cite{lorenz2013necessary}. %
But this is insufficient for well-defined regularisation via critical points, as we now illustrate.
\begin{example}
    Inspired by the example in \cite{obmann2022convergence}, we consider $\R(x)=|x|+\cos(x)$, %
    $\A=0$, $\alpha> 0$, and a sequence of measurements $y_k = 0$. In this case the sequence $x_k=2\pi k+\frac{\pi}{2}$ is such that each $x_k$ is a critical point of $\mathcal{J}_{\alpha, y_k}(x) := \alpha \mathcal{R}(x) + \mathcal{D}\left(\A x,y_k\right) 
    $, but $(x_k)$ has no convergent subsequence. Yet $|x|+\cos(x)$ is coercive, and can be checked to be weakly convex. 
Therefore, unlike in the global minimisers setting, regulariser coercivity does not guarantee stability or convergence. 
\end{example}

\subsection{Set-up and bounding the critical points}
 We are interested in Tikhonov functionals $\mathcal{J}_{\alpha, y^\delta}: \X \rightarrow [0,\infty)$. 
Instead of working with global minimisers of $\mathcal{J}_{\alpha, y^\delta}$, we consider regularised solutions $x_\alpha^\delta \in \operatorname{crit}\mathcal{J}_{\alpha, y^\delta}$ as critical points of $\mathcal{J}_{\alpha, y^\delta}$
We wish to construct non-convex regularisers which still achieve convergent regularisation in terms of critical points. By the discussion above, we will need to make assumptions on $\R$ which bound the set of critical points. %

    \begin{theorem}[Bounded critical points]\label{thm:bounded}
        Let $\R = R_{wc}+R_{sc}$ where $R_{wc}:\X\to[0,\infty)$ is  $\gamma$-weak convex and $R_{sc}$ is $\mu$-strongly convex. 
        For any $\hat{x}$ a critical point of $\R$, we have that for all $z\in\X$: If $\gamma<2\mu$, 
\begin{equation*}
 \|\hat{x}-z\|_{\X} \leqslant \frac{1}{(\mu-\frac{\gamma}{2})}\|\partial R_{sc}(z)\|+ \sqrt{\frac{1}{(\mu-\frac{\gamma}{2})}R_{wc}(z)},
\end{equation*}
           or if $R_{wc}$ is $L_R$-Lipschitz continuous, 
            \begin{equation*}
 \|\hat{x}-z\|_{\X}\leqslant \frac{1}{\mu}\left(L_R+\|\partial R_{sc}(z)\|\right).
\end{equation*}
        Here $\|\partial R_{sc}(z)\|:= \sup \{ \|\psi\|_{\X^*} : \psi \in \partial R_{sc}(z) \}$.
    \end{theorem}
    \begin{proof}
        Proven in \Cref{ap:proof_bounded}.
    \end{proof}
\begin{nb}
Note that $\R$ is $(\mu-\gamma)$-convex. Thus, $\R$ can posses multiple minimisers when $\mu\leqslant\gamma<2\mu$. %
This convexity bound is optimal, up to equality, and we provide an intuitive explanation for it, constructing an example with infinitely many unbounded minimisers in \Cref{ap:counterexample}. 
\end{nb}

\section{Inverse problem guarantees} \label{sec:InvP}
In this section, we show that weakly convex regularisation is well-defined, under the following assumptions, which are standard in the literature (see \cite{obmann2022convergence,grasmair2010generalized,poschl2008tikhonov}), except for \Cref{assumptions}(3b), which is imposed to avoid a counterexample (see \Cref{ap:infDcounterexample}) that arises in the infinite-dimensional setting. 

\begin{assumption}\label{assumptions}
\begin{itemize} 
\item[(1)] 
$\X$ is a reflexive Banach space.
\item[(2)] $\mathcal{R}$ is weakly sequentially l.s.c.
\item[(3)] $\mathcal{R}=R_{wc}+R_{sc}$, where $R_{wc}:\X\to[0,\infty)$ is  $\gamma$-weak convex, $R_{sc}$ is $\mu$-strongly convex, and either 
\begin{itemize}
    \item[(a)] $\gamma \leqslant \mu$ (i.e., $\R$ is convex), or
    \item[(b)] $\mu < \gamma < 2\mu$, and $R_{sc} - \frac12 \mu \|\cdot\|_{\X}^2$ is weakly sequentially l.s.c.
\end{itemize}  
This, in particular, means that $\mathcal{R}$ is coercive, and therefore so is $\mathcal{J}_{\alpha,y^\delta}$, since $\mathcal{D}$ is non-negative.
\item[(4)] $\mathcal{D}$ is  weakly sequentially l.s.c., convex in its first argument, continuous in its second argument, and $\D(y_1,y_2) = 0$ if and only if $y_1 = y_2$. 
\item[(5)] There exist $C>0$ and $p\geqslant1$ s.t. for all $y_1, y_2, y_3 \in \Y$, $ \mathcal{D}(y_1, y_2) \leqslant C\left(\mathcal{D}\left(y_1, y_3\right)+\|y_2- y_3\|_{\Y}^p\right)$.
\end{itemize}
\end{assumption}
\subsection{Existence and stability of solutions}

\begin{theorem}[Existence]\label{thm:existance} Under \Cref{assumptions} solutions exist, i.e. for all $\alpha>0$ and $y^\delta \in \Y$, $\operatorname{crit} \mathcal{J}_{\alpha, y^\delta}$ is non-empty.
\end{theorem}
\begin{proof}\vspace{-1em}
Existence of minimisers of $\mathcal{J}_{\alpha, y^\delta}$ follows from the coercivity and the continuity assumptions on $\mathcal{J}_{\alpha, y^\delta}$. 
\end{proof}
\begin{theorem}[Stability] \label{thm:stability} 
Let $y_k,y^\delta \in \Y$, $\alpha>0$, and $y_k \rightarrow y^\delta$ (in norm), and assume that $x_k \in \X$ is such that $x_k
\in \mathrm{crit}\,\, \mathcal{J}_{\alpha, y_k}$. Then under \Cref{assumptions} the sequence $\left(x_k\right)$ has a weakly convergent subsequence and the weak limit $x_{+}$ of any such subsequence is a critical point of $\mathcal{J}_{\alpha, y^\delta}$.
\end{theorem}
\begin{proof}\vspace{-1em}
    Proven in \Cref{ap:stabilityproof}. 
\end{proof}
In words: 
the reconstruction from a measurement at a given noise level (i.e., a critical point of the variational energy) is (weakly) continuous with respect to perturbations of the measurement (up to subsequences).

\subsection{Convergent regularisation}

\begin{theorem} \label{thm:convreg} Let \Cref{assumptions} hold. Let $y^0 \in \Y$ and assume that $y^0 = \A u$ for some $u \in \X$. Let $\delta_k \downarrow 0$, and choose $\alpha=\alpha(\delta)$ such that for $\alpha_k=\alpha\left(\delta_k\right)$ we have $\lim _k \alpha_k=\lim _k \delta_k^p / \alpha_k=0$, where $p$ is the same $p$ from \Cref{assumptions}(5).
Let $y^{\delta_k}\in \Y$  be any sequence of measurements with $\|y^{\delta_k} - y^0\|_{\Y} \leqslant \delta_k$, and let 
$x_k \in \X$ satisfy $x_k
\in \mathrm{crit}\,\, \mathcal{J}_{\alpha_k, y^{\delta_k}}$. Then the sequence $\left(x_k\right)$ has a weakly convergent subsequence and the weak limit $x_{+}$ of any such subsequence is an $\R$-criticising solution.
If there is a unique $u \in \X$ such that $\A u = y^0$, then $x_k \rightharpoonup u$. %
\end{theorem}
\begin{proof}\vspace{-1em}
    Proven in \Cref{ap:convregproof}. 
\end{proof}
In words: in the limit as the noise level vanishes, there is a regularisation parameter selection strategy under which reconstructions from measurements at each noise level converge to a solution of the noiseless operator equation, up to a subsequence unless such a solution is unique. %

\begin{nb}
\Cref{thm:stability,thm:convreg} do not hold in the general setting where $\R$ is weakly convex without \Cref{assumptions}(2,3), even assuming that $\R$ is globally Lipschitz; for a counterexample see \Cref{ap:infDcounterexample}. 
\end{nb}
\begin{nb}
    If $\X$ is a \textit{separable} reflexive Banach space, then the subsequences in 
\Cref{thm:stability,thm:convreg} can be explicitly constructed, as the Banach--Alaoglu theorem has a constructive proof in this special case. We omit the details.
\end{nb}
\begin{nb}\label{nb:fd}
    If we make the further assumption that $\X$ is a finite-dimensional Hilbert space then \Cref{thm:stability,thm:convreg} hold with strong convergence of the iterates, with \Cref{assumptions}(3) replaced with that $R_{wc}$ is globally Lipschitz or that $\gamma < 2\mu$. In the remainder of this paper, we will assume that $\X$ and $\Y$ are finite-dimensional real Hilbert spaces. 
\end{nb}

\section{Optimisation guarantees}\label{sec:Opt}

In this section, we analyse the primal-dual algorithm in the weakly convex setting. All results are in the setting of \Cref{assumptions} with the modification described in \Cref{nb:fd}.

\subsection{Primal-dual optimisation}
The idea of primal-dual optimisation is to reformulate \cref{eq:ful_invprob} as a minimax problem. First, we rewrite \cref{eq:ful_invprob} as 
\begin{equation}\label{eq:invprob2}
    \min_{x \in \X} \R(x) + F(\A x), 
\end{equation}
where $F(y) := \D(y,y^\delta) $. By \Cref{assumptions}(4), $F$ is convex and l.s.c.,
so by \Cref{thm:dualdual} we can rewrite \cref{eq:invprob2} as the minimax problem:
\[
\min _{x \in \X} \max _{y \in \Y} L(x,y):=\R(x)+ \langle \A x, y \rangle_{\Y}  -F^*(y).
\]
This can then be solved via the modified primal-dual hybrid gradient method (PDHGM) due to \cite{ChambollePock}. The method consists of the following updates for step sizes $\tau,\sigma > 0$:{
{
\begin{align} \label{eq:pdhgm_update}
x^{k+1}  \nonumber
&:=\argmin_{x\in \X} \left\{\R(x)+\left\langle y^k ,\A x\right\rangle_{\Y}+\frac{1}{2 \tau}\left\|x-x^k\right\|_{\X}^2\right\},\\ 
x_\vartheta^{k+1} & :=x^{k+1}+\vartheta\left(x^{k+1}-x^k\right), \\ \nonumber
y^{k+1} 
&:=\argmin_{y\in \Y} \left\{F^*(y)-\left\langle y, \A x_\vartheta^{k+1} \right\rangle_{\Y}+\frac{1}{2 \sigma}\left\|y - y^k\right\|_{\Y}^2\right\}.%
\end{align}
}}
Note that for non-convex $\R$, the first update may not be unique for all choices of $\tau$. However for $\rho$-weakly convex $\R$ it is unique for $\tau<1/\rho$, as shown in \Cref{thm:moreauenv_wc}.
This method, in the case of non-convexity and K{\L} (see \Cref{ap:KL}) was analysed for $\vartheta=0$ in \cite{SUN201815}. Here, we are interested in $\vartheta=1$ similar to \cite{ChambollePock}, due to connections to the proximal point method and the alternating direction method of multipliers \cite{lu2023unified}. We will use weak convexity of $\R$ and $\mu$-strong convexity of $F^*$, satisfied, e.g., when $F$ is convex and $1/\mu$-smooth.
\subsection{Convergence of PDHGM}
First we rewrite our problem in a nice form in analogy with \cite{lu2023unified}. Let
 $z:=(x, y)$ %
 and %
 define \[
T(z):=\begin{pmatrix}\partial_x L(x, y) \\ -\partial_y L(x, y)\end{pmatrix} = \begin{pmatrix}
    \partial\R(x) + \A^* y\\ 
    \partial F^*(y) - \A x
\end{pmatrix} 
,\]
where $\A^*:\Y \to \X$ is the adjoint of $\A$, i.e. for all $x \in \X$ and $y \in \Y$, $\langle \A x, y \rangle_{\Y} = \langle x, \A^* y\rangle_{\X}$.  
Setting $z^k :=(x^k,y^k)$, the update rule of PDHGM can be written %
as
\begin{equation*}
M\left(z^k-z^{k+1}\right) \in T\left(z^{k+1}\right), \:\: \text{ for } \:\:
M := \begin{pmatrix}
\frac{1}{\tau} I & -\A^*\\
-\vartheta\A & \frac{1}{\sigma} I
\end{pmatrix}.
\end{equation*}

For $\vartheta=1$, $M$ is a self-adjoint positive semi-definite operator if $\tau\sigma\|\A\|^2<1$. Henceforth, we only consider $\vartheta = 1$.
We define the following Lyapunov function, whose critical points coincide with those of $L$:
\begin{equation}\label{eq:lyapunov}
\mathcal{L}(z, z'):=L(z)+\frac12\|z'-z\|_{M}^2,
\end{equation}
where 
 \[\|z\|^2_M := \left\langle z,Mz \right\rangle_{\X\times\Y} %
= \frac{1}{\tau}\|x\|_{\X}^2 - 2 \langle \A x, y\rangle_{\Y} + \frac{1}{\sigma}\|y\|^2_{\Y}
.\] For $x \in \X$ and $y \in \Y$, we denote $\|y\|_M:=\|\left(0,y \right)\|_M = \frac1{\sqrt{\sigma}}\|y\|_{\Y}$, and similarly  $\|x\|_M:=\|\left(x,0\right)\|_M=\frac1{\sqrt{\tau}}\|x\|_{\X}$. %
\begin{theorem}[Strict descent]\label{thm:descent} 
For $(x^k,y^k)$ satisfying the PDHGM updates \Cref{eq:pdhgm_update} with $\vartheta = 1$, $\R$ $\wcv{\rho}$-weak convex, $F^*$ $\scv{\mu}$-strong convex, and step sizes satisfying $\tau\sigma\|\A\|^2<1$, $\tau\wcv{\rho}<1$, and $\mu\scv{\sigma}>3$, the following descent holds for the Lyapunov function \Cref{eq:lyapunov}:
\begin{align*}
\mathcal{L}(z^{k}, z^{k-1})-\mathcal{L}(z^{k+1}, z^{k})\geqslant \frac{1}{2}(\scv{\mu}\sigma-3)\left\|y^{k}-y^{k+1}\right\|_{M}^2 +\frac{1}{2}(1-\wcv{\rho}\tau)\left\|x^{k}-x^{k+1}\right\|_{M}^2  %
\end{align*}
\end{theorem}
\begin{proof}
    Proven in \Cref{ap:descent}.
\end{proof}
Using this, we can prove the following under iterate boundedness, which is a standard assumption in analysis of non-convex optimisation \cite{SUN201815,bolte2014proximal}.
\begin{theorem}\label{thm:pd_results}
Assume that $\inf_{x\in \X} L(x,y) > -\infty$ for all $y\in \Y$, the $z^k = (x^k,y^k)$ are bounded, $\R$ is $\wcv{\rho}$-weak convex, $F^*$ is $\scv{\mu}$-strong convex, $\tau\sigma\|\A\|^2<1$, $\tau\wcv{\rho}<1$, and $\mu\scv{\sigma}>3$. Let $\nu := \min\{\scv{\mu}\sigma-3,1-\wcv{\rho}\tau\}$, then 
\[
\min _{k\leqslant K} \operatorname{dist}\left(0, \partial L(z^k)\right) \leqslant \frac{2}{\sqrt{\nu K}} \sqrt{\mathcal{L}(z^1,z^0) - \mathcal{L}(z^{K+1},z^{K})}.
\]
Furthermore, the iterates are square-summable: 
\begin{align*}
\sum_{k=1}^{\infty}\left\|x^{k+1}-x^k\right\|_{\X}^2<\infty, && \sum_{k=1}^{\infty}\left\|y^{k+1}-y^k\right\|_{\Y}^2<\infty.
\end{align*}
Letting $\mathcal{C}$ denote the set of cluster points of $z^k$, we have that $\mathcal{C}$ is a nonempty compact set and 
$$
\lim _{k \rightarrow \infty} \operatorname{dist}\left(\left(x^k, y^k\right), \mathcal{C}\right)=0, 
$$
with $\mathcal{C} \subseteq \operatorname{crit} L$ and $L$ is finite and constant on $\mathcal{C}$.
\end{theorem}
\begin{proof}\vspace{-1em}
    Proven in \Cref{as:pd_results}.
\end{proof}
\begin{nb}
It is important to note that parameter bounds $\tau\sigma\|\A\|^2<1$, $\tau\wcv{\rho}<1$, and $\mu\scv{\sigma}>3$ are likely not optimal. It is, however, always possible to choose step sizes to satisfy these, taking $\tau < \min\left\{\frac{1}{\rho},\frac{\mu}{3\|\A\|^2}\right\}$ and $\sigma \in \left(\frac{3}{\mu},\frac{1}{\tau\|\A\|^2}\right)$.
\end{nb}
\subsection{Ergodic convergence rate}\label{sec:ergodic}
We now prove ergodic convergence rates, bounding the primal-dual gap for the PDHGM. To derive these, we will need bounds on iterate distance from a general point. These are not entailed by weak convexity, so we make the further assumption that the Lyapunov function $\mathcal{L}$ has the Kurdyka--\L ojasiewicz property (see \Cref{ap:KL} for details). By the following note, this assumption is not too restrictive. 
\begin{nb}
By \Cref{thm:NNKL}, if $\R$ is a deep neural network with continuous, piecewise analytic activations with finitely many pieces (e.g., \texttt{ReLU}), then $\R$ is subanalytic.
If $F(y) := \frac{1}{2\alpha}\|y - y^\delta\|_{\Y}^2$ then all the remaining terms in $\mathcal{L}$ are analytic. It follows that $\mathcal{L}$ is subanalytic (by e.g. Lemma 7.4 of \cite{BuddJRS}) and therefore is K{\L} with K{\L} exponent in $[0,1)$ on all of its domain \cite{bolte2007lojasiewicz}. 
\end{nb}
The following results are standard in the literature, but will prove useful for proving ergodic convergence rates. 
\begin{theorem}\label{thm:KLconvergence}
Suppose that $\mathcal{L}$ is K\L~ and the $z^k$ are bounded, then $z^k$ converges to a critical point $\hat z$ of $L$ and
$$
\sum_{k=1}^{\infty}\left\|z^{k+1}-z^k\right\|_M<\infty.
$$    
Furthermore, if $\mathcal{L}$ has K{\L} exponent $\theta\in[0,1)$ at $\hat z$, we have that there exist constants $\nu_1,\nu_2>0$ and $ \tau\in(0,1)$, s.t.:\\
If $\theta=0$, the sequence $z^k$ converges in finite steps.\\
If $\theta \in\left(0, \frac{1}{2}\right]$, $\left\|z^k-\hat{z}\right\|_M \leqslant \nu_1 \tau^{k}.$\\
If $\theta \in\left(\frac{1}{2}, 1\right)$, $\left\|z^k-\hat{z}\right\|_M \leqslant \nu_2 k^{-\frac{1-\theta}{2 \theta-1}}$.
\end{theorem}
\begin{proof} \vspace{-1em}%
See
\cite{guo2023preconditioned,SUN201815}.
\end{proof}
This lets us derive the corresponding ergodic rate. 
\begin{theorem}\label{thm:ergodic_rate}
Assume that the bounded sequence $z^k=(x^k,y^k)$ are PDHGM iterates from \Cref{eq:pdhgm_update} with $\vartheta = 1$, $\R$ $\wcv{\rho}$-weak convex and $F^*$ $\scv{\mu}$-strong convex, satisfying $\tau\sigma\|\A\|^2<1$, $\tau\wcv{\rho}<1$, and $\mu\scv{\sigma}>3$. Assume that  $\mathcal{L}$ is K{\L}, with K{\L} exponent $\theta\in[0,1)$ at $\hat z$, denoting the limit $z^k \to \hat{z} = (\hat x,\hat y)$, and let $\bar{x}^k$ and $\bar{y}^k$ denote the means of $(x^i)_{i=1}^k$ and of $(y^i)_{i=1}^k$, respectively. %
Then, for all $x\in \X$ and $y \in \Y$:
\begin{equation*}
 L\left(\bar{x}^k, y\right)-L\left(x, \bar{y}^k\right) \leqslant \frac{1}{2}\left(\rho\|x-\hat{x}\|_{\X}^2 - \mu \|y-\hat{y}\|_{\Y}^2 \right)+\mathcal{O}\left(\frac{\log k}{k}\right).   
\end{equation*}
\end{theorem}
\begin{proof}
    Proven in \Cref{ap:ergodic}.
\end{proof}
\begin{nb}
   In the case of $\rho=\mu=0$ we recover the usual convex ergodic rates \cite{lu2023unified}.
    Unlike the convex case, here there is a non-$k$ dependent term, which does not tend to zero. This arises due to non-convexity, and hence non-uniqueness of critical points of the original problem.
\end{nb}
\begin{nb}
   We have here only shown iterate convergence in the case of $\R$ $\rho$-weak convex and $F^*$ $\mu$-strong convex, but the proof of this result naturally extends to the case of $F^*$ being weak convex, as long as convergence can be shown. 
\end{nb}

\section{Learning a weakly convex regulariser}
Input weakly convex neural networks (IWCNNs) were introduced in \cite{shumaylov2023provably} with the following idea. As discussed in \Cref{sec:intro11}, non-convex regularisers have been observed to result in better performance than convex regularisers, and `optimal' regularisers are typically non-convex yet often weakly convex. We therefore want a neural network which is weakly convex but not convex. A smooth neural network achieves this, but these are usually undesirable \cite{krizhevsky2017imagenet}. But a neural network with all \texttt{ReLU} activations is weakly convex if and only if it is convex \cite{shumaylov2023provably}. To get the best of both worlds, the IWCNN architecture uses the following result: 
if  $f = g_c \circ g_{sm}$ for $g_c: \mathbb{R}^n \rightarrow \mathbb{R}$ convex and $L$-Lipschitz and $g_{sm}: \mathbb{R}^k \rightarrow \mathbb{R}^n$ $C^1$ with $\beta$-Lipschitz gradient, then $f$ is $L \beta$-weakly convex \cite{davis2018subgradient}.
\cite{shumaylov2023provably} utilises an input convex neural network (ICNN) \cite{ICNN} to define an IWCNN.
\begin{definition}[IWCNN; Definition 4.1 in \cite{shumaylov2023provably}]\label{def:IWCNN}
An IWCNN $f^\textrm{IWCNN}_\theta$ is defined by 
\[
f^\textrm{IWCNN}_\theta :=g^\textrm{ICNN}_{\theta_1}\circ g^\textrm{sm}_{\theta_2},\] where $g^\textrm{ICNN}_{\theta_1}$ is an input convex neural network and $g^\textrm{sm}_{\theta_2}$ is a neural network with Lipschitz smooth activations. %
\end{definition}

We now prove that IWCNNs can universally approximate continuous functions. We will use the fact (see \cite{sun2019least} Theorem 6) that the weakly convex functions (on an open or closed convex domain) are dense in the continuous functions (on that domain) with respect to $\|\cdot\|_\infty$. 
Therefore, as long as IWCNNs can approximate any weakly convex function, they can approximate any continuous function.

\begin{theorem}[Universal Approximation of IWCNN]\label{thm:universal_IWCNN}
    Consider an IWCNN as in \Cref{def:IWCNN}, with intermediate dimension $m\in\mathbb{N}$. For any $m$, IWCNNs can uniformly approximate any continuous function on a compact domain.
\end{theorem}
\begin{proof}\vspace{-1em}
    Proof provided in \Cref{ap:universal_IWCNN}.
\end{proof}
\subsection{Adversarial regularisation}
\begin{figure*}[t]
  \centering
  \begin{minipage}[t]{0.2\linewidth}
  \centering
  \vspace{0pt}
  \begin{tikzpicture}[spy using outlines={circle,red,magnification=3.0,size=1.0cm, connect spies}]   
    \node {\includegraphics[width=\linewidth]{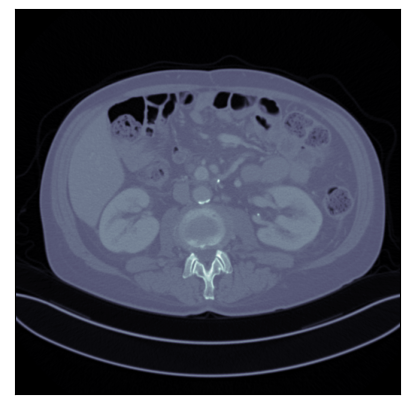}};
    \spy on (0.035,-0.65) in node [left] at (1.7,1.25);
  \end{tikzpicture}
  \vskip-0.5\baselineskip
  {\scriptsize Ground-truth}
  \end{minipage}\hfill
  \begin{minipage}[t]{0.2\linewidth}%
  \centering
  \vspace{0pt}
  \begin{tikzpicture}[spy using outlines={circle,red,magnification=3.0,size=1.0cm, connect spies}]   
    \node {\includegraphics[width=\linewidth]{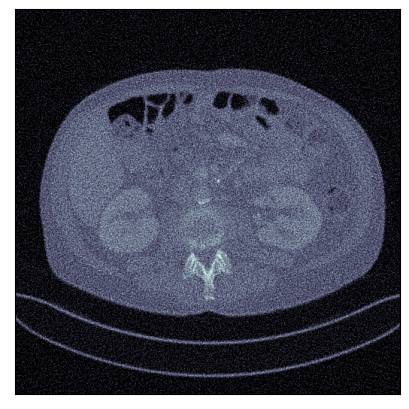}};
    \spy on (0.035,-0.65) in node [left] at (1.7,1.25);  
  \end{tikzpicture}
  \vskip-0.5\baselineskip
  {\scriptsize FBP: 21.303 dB, 0.195}
  \end{minipage}\hfill
  \begin{minipage}[t]{0.2\linewidth}%
  \centering
  \vspace{0pt}
  \begin{tikzpicture}[spy using outlines={circle,red,magnification=3.0,size=1.0cm, connect spies}]   
    \node {\includegraphics[width=\linewidth]{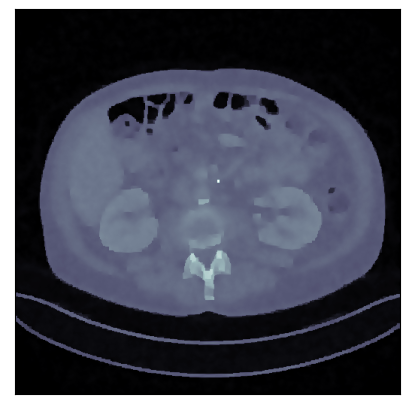}};
    \spy on (0.035,-0.65) in node [left] at (1.7,1.25);
  \end{tikzpicture}
  \vskip-0.5\baselineskip
  {\scriptsize TV: 31.690 dB, 0.889}
  \end{minipage}\hfill
  \begin{minipage}[t]{0.2\linewidth}%
  \centering
  \vspace{0pt}
  \begin{tikzpicture}[spy using outlines={circle,red,magnification=3.0,size=1.0cm, connect spies}]   
    \node {\includegraphics[width=\linewidth]{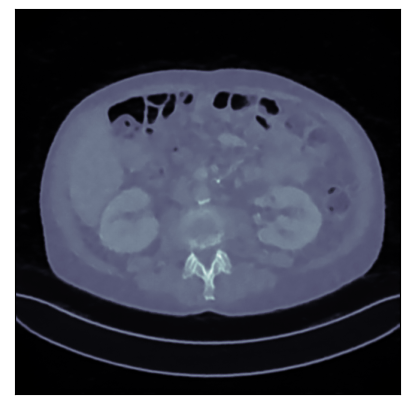}};
 \spy on (0.035,-0.65) in node [left] at (1.7,1.25); 
  \end{tikzpicture}
  \vskip-0.5\baselineskip
  {\scriptsize U-Net: 36.712 dB, 0.920}
  \end{minipage}\hfill
  \begin{minipage}[t]{0.2\linewidth}%
  \centering
  \vspace{0pt}
  \begin{tikzpicture}[spy using outlines={circle,red,magnification=3.0,size=1.0cm, connect spies}]   
    \node {\includegraphics[width=\linewidth]{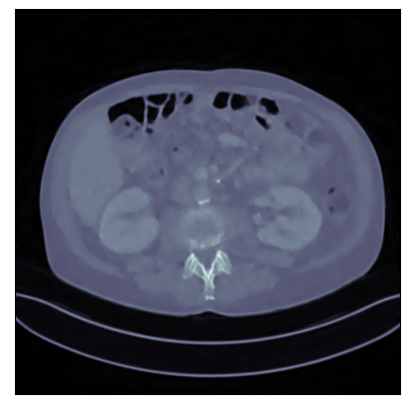}};
    \spy on (0.035,-0.65) in node [left] at (1.7,1.25); 
  \end{tikzpicture}
  \vskip-0.5\baselineskip
  {\scriptsize LPD: 36.810 dB, 0.912}
  \end{minipage}\hfill
  \begin{minipage}[t]{0.2\linewidth}%
  \centering
  \vspace{0pt}
  \begin{tikzpicture}[spy using outlines={circle,red,magnification=3.0,size=1.0cm, connect spies}]   
    \node {\includegraphics[width=\linewidth]{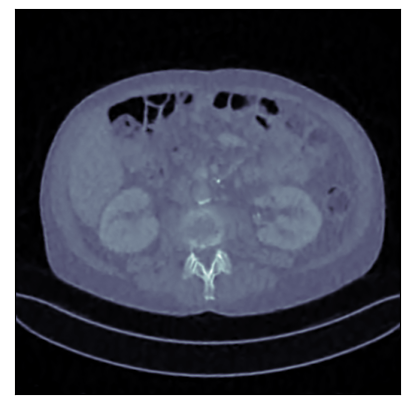}};
    \spy on (0.035,-0.65) in node [left] at (1.7,1.25); 
  \end{tikzpicture}
  \vskip-0.5\baselineskip
  {\scriptsize AR: 36.694 dB, 0.907}
  \end{minipage}\hfill
  \begin{minipage}[t]{0.2\linewidth}%
  \centering
  \vspace{0pt}
  \begin{tikzpicture}[spy using outlines={circle,red,magnification=3.0,size=1.0cm, connect spies}]   
    \node {\includegraphics[width=\linewidth]{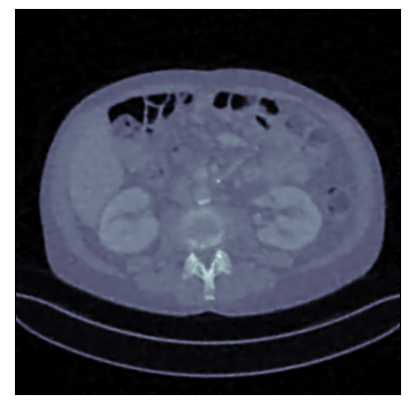}};
    \spy on (0.035,-0.65) in node [left] at (1.7,1.25); 
  \end{tikzpicture}
  \vskip-0.5\baselineskip
  {\scriptsize ACR: 35.708 dB, 0.897}
  \end{minipage}\hfill
   \begin{minipage}[t]{0.2\linewidth}%
  \centering
  \vspace{0pt}
  \begin{tikzpicture}[spy using outlines={circle,red,magnification=3.0,size=1.0cm, connect spies}]   
    \node {\includegraphics[width=\linewidth]{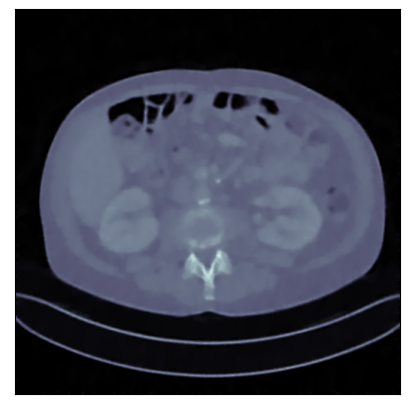}};
    \spy on (0.035,-0.65) in node [left] at (1.7,1.25); 
  \end{tikzpicture}
  \vskip-0.5\baselineskip
  {\scriptsize ACNCR: 36.533 dB, 0.921}
  \end{minipage}\hfill
   \begin{minipage}[t]{0.2\linewidth}%
  \centering
  \vspace{0pt}
  \begin{tikzpicture}[spy using outlines={circle,red,magnification=3.0,size=1.0cm, connect spies}]   
    \node {\includegraphics[width=\linewidth]{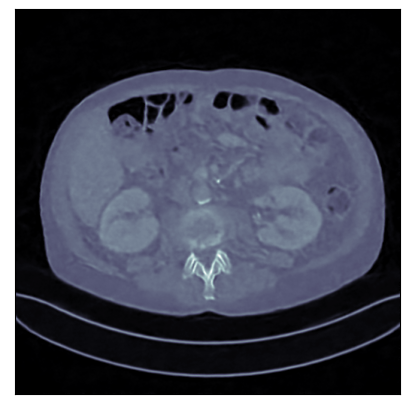}};
    \spy on (0.035,-0.65) in node [left] at (1.7,1.25);  
  \end{tikzpicture}
  \vskip-0.5\baselineskip
  {\scriptsize AWCR: 37.603 dB, 0.918}
  \end{minipage}\hfill
   \begin{minipage}[t]{0.2\linewidth}%
  \centering
  \vspace{0pt}
  \begin{tikzpicture}[spy using outlines={circle,red,magnification=3.0,size=1.0cm, connect spies}]   
    \node {\includegraphics[width=\linewidth]{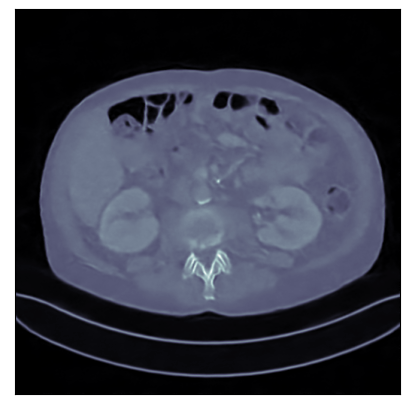}};
    \spy on (0.035,-0.65) in node [left] at (1.7,1.25);
  \end{tikzpicture}
  \vskip-0.5\baselineskip
  {\scriptsize AWCR-PD: 37.941 dB, 0.924}
  \end{minipage}\hfill
  \caption{\small{Reconstructed images obtained using different methods, along with the associated PSNR and SSIM, for sparse view CT. In this case the AWCR and AWCR-PD achieve the highest PSNR and SSIM. Furthermore, both AWCR methods retain the fine-structure in the reconstruction, unlike the ACNCR and ACR, the only other methods which possess convergence guarantees.} \label{fig:visual}}
\end{figure*}

 Adversarial learning operates within the weakly supervised setting described in \Cref{sec:intro}. 
We have datasets $(x_i) \in \X$ and $(y_j) \in \Y$ each i.i.d. sampled from the distributions of ground truth images $\mathbb{P}_r$ and measurements $\mathbb{P}_Y$, respectively. To get these two distributions into the same space, we push-forward $\mathbb{P}_Y$ from the measurement space $\Y$ to the image space $\X$ using some pseudo-inverse $\A^\dagger$ of the forward model, giving $\mathbb{P}_n:=(\A^{\dagger})_{\#} \mathbb{P}_Y$, a distribution of images with reconstruction artefacts. 

The key idea is that regularisation is really a type of classification problem. We want $\R$ to be small on the ground truth images (i.e., $\mathbb{P}_r$) and large on artificial images (i.e., $\mathbb{P}_n$). Therefore, \cite{AR} learned a neural network regulariser $\R_\theta$ by minimising the following loss functional:
\begin{equation}
    \label{eq:AR_loss}
        \mathbb{E}_{X \sim \mathbb{P}_r}\left[\R_{\theta}(X)\right]-\mathbb{E}_{X \sim \mathbb{P}_n}\left[\R_{\theta}(X)\right]+\lambda \cdot \mathbb{E}\left[\left(\left\|\partial_x \R_{\theta}(X)\right\|-1\right)_{+}^2\right].
\end{equation}
The final term pushes $\R_\theta$ to be 1-Lipschitz, inspired by the Wasserstein GAN (WGAN) loss \cite{arjovsky2017wasserstein}, and the expectation is taken over all lines connecting samples in $\mathbb{P}_n$ and $\mathbb{P}_r$. 
This $\R_{\theta}$ has a key interpretation, given two assumptions. First, assume that the measure $\mathbb{P}_r$ is supported on the weakly compact set $\mathcal{M}$. %
This captures the intuition that real image data lies in a lower-dimensional non-linear subspace of the original space. 
Let  $P_{\mathcal{M}}$ be the projection function onto $\mathcal{M}$, assumed to be defined $\mathbb{P}_n$-a.e. Then second, assume that the measures $\mathbb{P}_r$ and $\mathbb{P}_n$ satisfy $
\left(P_{\mathcal{M}}\right)_{\#}\left(\mathbb{P}_n\right)=\mathbb{P}_r
$. %
This essentially means that the reconstruction %
artefacts are small enough to allow recovery of the real distribution by simply projecting the noisy distribution onto $\mathcal{M}$. Of the two assumptions, the latter is stronger.
\begin{theorem}[Optimal AR \cite{AR}]\label{thm:optimal_AR}
Under these assumptions, the distance function to $\mathcal{M}$, $
d_{\mathcal{M}}(x):=\min _{z \in \mathcal{M}}\|x-z\|_{\X}
$ is a maximiser of 
\begin{equation*}
\label{eq:Wass}
\sup _{f \in \text{1-Lip}} \mathbb{E}_{X \sim \mathbb{P}_n}[f(X)]-\mathbb{E}_{X \sim \mathbb{P}_r}[f(X)].
\end{equation*}
\end{theorem}
\begin{nb}[Non-uniqueness]
The functional in \Cref{thm:optimal_AR} does not necessarily have a unique maximiser. 
For a more in-depth discussion of when potentials and transports {can} be unique, see \cite{staudt2022uniqueness,milne2022new}.
\end{nb}

In \cite{ACR} this approach was modified by parameterising  $\mathcal{R}_{\theta}(x)=\mathcal{R}^\textrm{ICNN}_\theta(x)+\frac{\mu_0}{2}\|x\|_{\X}^2$, where $\mathcal{R}^\textrm{ICNN}_\theta$ is an ICNN, and then minimising \cref{eq:AR_loss} to learn an adversarial convex regulariser (ACR). In order to retain the interpretation of \Cref{thm:optimal_AR}, we would have to assume that $d_\mathcal{M}$ is convex (so that it can be approximated by an ICNN), which is true if and only if $\mathcal{M}$ is a convex set. However, real data rarely lives on a convex set.  This breakdown can lead to performance issues, as illustrated on \Cref{fig:toy_example}, where an ICNN completely fails to approximate the distance function.

\subsection{Adversarial weakly convex and convex-nonconvex regularisation }

In \cite{shumaylov2023provably}, an adversarial convex-nonconvex regulariser (ACNCR) was defined by $\R_\theta(\x) := \R_{\theta_1}^{c}(\x) + \R_{\theta_2}^{wc}(\A x)$, for $\R_{\theta_1}^{c}$ parameterised the same as the ACR, and $\R_{\theta_2}^{wc}$ parameterised with an IWCNN. Then $\R_{\theta_1}^{c}$ and $\R_{\theta_2}^{wc}$ are trained in a decoupled way, with $\R_{\theta_1}^{c}$ minimising \cref{eq:AR_loss} and $\R_{\theta_2}^{wc}$ minimising a similar loss but with $\mathbb{P}_n$ replaced by $\mathbb{P}_Y$ and $\mathbb{P}_r$ replaced by
$\mathbb{P}_{Y_r} :=\left(\A\right)_{\#} \mathbb{P}_r$, the push-forward of the ground truth distribution under the forward model. This was an important step towards greater interpretability, since, by \Cref{thm:optimal_AR}, the optimal $\R_{\theta_2}^{wc}(y) \approx d_{\A[\mathcal{M}]}(y)$, but it is still unclear how to interpret $\R_{\theta_1}^{c}$. 
In this work, we define a more directly weak convex parameterisation, following the ACR.
\begin{definition}[Adversarial weak convex regulariser]\label{def:AWCR}

The \textit{adversarial weak convex regulariser} (AWCR) is parameterised by $\R_\theta(x) =\mathcal{R}^\textrm{IWCNN}_\theta(x)+\frac{\mu_0}{2}\|x\|_{\X}^2$, where $\mathcal{R}^\textrm{IWCNN}_\theta$ is an IWCNN. It is learned by minimising \cref{eq:AR_loss}. This satisfies \Cref{assumptions} with the modification described in \Cref{nb:fd}, as an IWCNN is Lipschitz by construction, and so $\mu$ can be chosen to be arbitrarily small.
\end{definition}
\begin{nb}
\Cref{thm:universal_IWCNN} allows this AWCR to overcome the ACR's issue of ICNNs not being able to approximate a distance function to non-convex sets. An IWCNN is able to approximate the distance function to any compact set, since any such distance function is 1-Lipschitz continuous. 
\end{nb}
\section{Numerical results}\label{sec:results}

\subsection{Distance function approximations}
To illustrate the importance of proper neural network structure, we first discuss a toy example. 
We consider data lying on a product manifold $\mathcal{M}$ of two separate spirals embedded into $\mathbb{R}^2$ in the form of a double spiral, as shown in %
\Cref{fig:toy_example}. We then train an AR, an ACR (i.e., an ICNN), and an AWCR (i.e., an IWCNN) on  the denoising problem for this data, %
and compare the learned regularisers with the true $d_\mathcal{M}$. 
As is clear from \Cref{fig:toy_example}, the ICNN entirely fails to approximate this non-convex $d_{\mathcal{M}}$. This limitation is overcome by the IWCNN, which despite the imposed constraints can approximate the true function well, and even shows improved extrapolation. 
Despite the simplicity of this example, the overall hierarchy of methods persists in the experiments explored in \Cref{sec:CT} on higher dimensional real data.

\subsection{Computed Tomography (CT)}\label{sec:CT}
For evaluation of the proposed methodology, we consider two applications: CT reconstruction with (i) sparse-view and (ii) limited-angle projection. Supervised methods require access to large high-quality datasets for training, but outside of curated datasets, obtaining large amounts of high-quality paired data is unrealistic. For this reason weakly supervised methods, requiring access only to unpaired data, are of significant interest for this problem.

We consider two versions of the AWCR method, one using the subgradient method to solve \cref{eq:ful_invprob} as in \cite{AR,ACR,mukherjee2024data}, and one using PDHGM, denoted AWCR-PD. These are compared with: 
\begin{itemize}
\item two standard knowledge-driven techniques: %
filtered back-projection (FBP) and total variation (TV) regularisation, which act as the baseline; 
\item two supervised data-driven methods: the learned primal-dual (LPD) method \cite{lpd_tmi} and U-Net-based post-processing of FBP \cite{postprocessing_cnn}, considered to be state of the art methods in end-to-end learned reconstruction from the two main paradigms: algorithm unrolling and learned post-processing; 
\item three weakly supervised methods: the AR, ACR, and the ACNCR. Adversarial regularisation methods were chosen as the currently best performing, to the authors' knowledge, weakly supervised methods for CT.
\end{itemize}
These comparisons illustrate the trade-offs in levels of constraints and supervision versus stability and performance. For details of the experimental set-up, see \Cref{ap:expsetup}. We measure the performance in terms of the peak signal-to-noise ratio (PSNR) and the structural similarity index (SSIM) \cite{ssim_paper_2004}. We report average test dataset results in \Cref{sample-table}, with further visual examples in \Cref{fig:visual}.

\textbf{Sparse view CT}\quad As in \cite{AR} performance of AR during reconstruction begins to deteriorate if the network is over-trained, so early stopping must be employed in training. For the ACR, ACNCR, and both AWCR methods this does not occur due to reduced expressivity, yet both AWCR methods surpass the performance of AR. Furthermore, the AWCR methods perform qualitatively better than the AR, ACR, and ACNCR (especially the latter two) at reconstructing fine details, which in the medical setting can be more important than PSNR/SSIM accuracy. %
Indeed, the AWCR-PD method approaches the PSNR accuracy of the strongly supervised U-Net post-processing method. 

\textbf{Limited view CT}\quad In this setting a specific angular region contains no measurement, turning this into a severely ill-posed inverse problem, and a good image prior is crucial for reconstruction. As shown in \cite{ACR}, the AR begins introducing artifacts during reconstruction, which is overcome for both the ACNCR and ACR due to the imposed convexity. The AWCR, on the other hand, is able to remain non-convex without experiencing deterioration, and slightly outperforms both the ACR and ACNCR in PSNR, but not SSIM. However, the AWCR-PD method performs worse in this setting, though still outperforming AR. This occurs due to the forward and the adjoint operator being severely ill-posed. For a visual comparison of the AWCR and AWCR-PD methods in this setting, see \Cref{ap:datavis}.

\begin{table}
  \caption{Average test PSNR and SSIM in CT experiments.}
  \label{sample-table}
  \centering
  \resizebox{.5\columnwidth}{!}{%
  \begin{tabular}{l c c c c}
    \toprule
    & \multicolumn{2}{c}{Limited} & \multicolumn{2}{c}{Sparse}                   \\
    \cmidrule(r){2-3} \cmidrule(r){4-5}
    Methods  & PSNR (dB)       & SSIM        & PSNR (dB)       & SSIM        \\
    \midrule
    & \multicolumn{4}{c}{\textit{Knowledge-driven}}                  \\
    FBP & 17.1949  & 0.1852  & 21.0157  & 0.1877  \\
    TV & \textbf{25.6778}  & \textbf{0.7934}  & \textbf{31.7619}  & \textbf{0.8883}  \\
    \midrule
    & \multicolumn{4}{c}{\textit{Supervised}}                  \\
    LPD &  28.9480 & \textbf{0.8394}  & \textbf{37.4868} & 0.9217  \\
    U-Net &  \textbf{29.1103} & 0.8067  &  37.1075 &  \textbf{0.9265}  \\
    \midrule
    & \multicolumn{4}{c}{\textit{Weakly Supervised}}                  \\
    AR & 23.6475 & 0.6257  & 36.4079  & 0.9101  \\
    ACR &  26.4459 & \textbf{0.8184}  & 34.5844 & 0.8765  \\
    ACNCR & 26.5420 & 0.8161  & 35.6476  &  0.9094  \\
    AWCR &  \textbf{26.7233} & 0.8105  & 36.5323  &  0.9046 \\
    AWCR-PD & 24.3254  & 0.7650 &  \textbf{36.9993} & \textbf{0.9108} \\
    \bottomrule
  \end{tabular}}
\end{table}

\section{Conclusions}
In this work, we have shown that the framework of weak convexity is fertile ground for proving theoretical guarantees. We proved that weakly convex regularisation can yield existence, stability, and convergence in the sense of critical points. Furthermore, we filled a gap in the literature by proving convergence (with an ergodic rate) for the PDHGM optimisation method (with $\vartheta=1$) in the weakly convex setting. All of these results hold when $\R$ is a weakly convex neural network. We showed that the IWCNN architecture can universally approximate continuous functions, allowing the AWCR to retain interpretability as a distance function (empirically illustrated in \Cref{fig:toy_example}) whilst also benefiting from all of the provable guarantees from weak convexity. Finally, for the problem of CT reconstruction we showed that the AWCR can compete with or outperform the state-of-the-art in adversarial regularisation, and approach the lower bound performance of state-of-the-art supervised approaches for CT, despite being merely weakly supervised.

Although we have focused on adversarial regularisation in this work, we emphasise that other learned regularisation methods can be put into this weakly convex setting. For example, in \cite{Hurault23} it was shown that plug-and-play proximal gradient descent approaches (using sufficiently smooth gradient step denoisers) correspond to criticising an energy with a weakly convex regulariser. The total deep variation regulariser \cite{TDV2020}
\[
\R(x) := \sum_{i=1}^n \sum_{j=1}^\ell \mathcal{N}_{ij}(Kx) w_j
\]
is weakly convex so long as $w_j \geqslant 0$ for all $j$ and the $\mathcal{N}_{ij}$ are weakly convex for all $i,j$, which is achieved if the U-Net $\mathcal{N}:\mathbb{R}^d \to \mathbb{R}^{n\ell}$ has the form of an IWCNN, using a uniformly convex U-Net as in \cite{Bianchi_2023}. The NETT $\ell^q$ regulariser \cite{li2020nett} has the form:
\[
\R(x):= \sum_{\lambda \in \Lambda} \|\Phi_{\lambda,\theta}(x)\|^q_q, 
\] 
where $q \geqslant 1$ and $\Lambda$ is a finite set. 
For $q = 2$ and $|\Lambda|=1$, this is also the form of the learned noise reduction regulariser used in \cite{Liu2022}. A sufficient condition for this $\R$ to be  weakly convex is, for all $\lambda$, $\Phi_{\lambda,\theta}$ is at least one of: bounded and $C^1$ with Lipschitz gradient (by \cite{davis2018subgradient}), or weakly convex (e.g., an IWCNN) with $(\Phi_{\lambda,\theta}(x))_i\in [0,B_i] $ for all $x$ and $i$ (by \Cref{prop:compwkcvx}).

A number of open questions remain. Expressivity questions about IWCNNs beyond universal approximation, e.g. efficiency of representation, remain unanswered. %
It remains open whether the PDHGM algorithm can be shown to be convergent for a generic weak-weak splitting; or whether the boundedness assumption can be dropped in certain cases. Alternatively, given the structure of the regulariser, it may be beneficial to turn to prox-linear schemes instead, analysis of which is currently missing in the primal-dual setting. Also, the problem of extracting proximal operators directly from learned networks remains open. Finally, there is the broader question of comparing the effectiveness of different approaches for learned regularisation, especially (as noted by an anonymous reviewer) quantifying how this effectiveness depends on the ill-posedness of the inverse problem.
\section*{Acknowledgements}
JB is supported by start-up funds at the California Institute of Technology. CBS acknowledges support from the Philip Leverhulme Prize, the Royal Society Wolfson Fellowship, the EPSRC advanced career fellowship EP/V029428/1, the EPSRC programme grant EP/V026259/1, and the EPSRC grants EP/S026045/1 and EP/T003553/1, EP/N014588/1, EP/T017961/1, the Wellcome Innovator Awards 215733/Z/19/Z and 221633/Z/20/Z, the European Union Horizon 2020 research and innovation programme under the Marie Skodowska-Curie grant agreement No. 777826 NoMADS, the Cantab Capital Institute for the Mathematics of Information and the Alan Turing Institute.
This research was supported by the NIHR Cambridge Biomedical Research Centre (NIHR203312). The views expressed are those of the author(s) and not necessarily those of the NIHR or the Department of Health and Social Care.

\bibliography{sample}
\bibliographystyle{unsrt}

\newpage
\appendix
\onecolumn
\section{Weakly convex analysis review}
\label{ap:convanalysis}

For a locally Lipschitz $f: \X \to\mathbb{R}$, the following are equivalent to $\rho$-convexity \cite{ambrosio2005gradient}: %
\begin{enumerate}
    \item For all $ x_1, x_2 \in \X$ and $\lambda \in[0,1]$,
    \begin{align*}
    f\left(\lambda x_1+(1-\lambda) x_2\right) \leqslant \lambda &f\left(x_1\right)+(1-\lambda) f\left(x_2\right) -\frac{\rho\lambda(1-\lambda)}{2} \left\|x_1-x_2\right\|_{\X}^2 .   
    \end{align*}
    \item For any $x, \hat{x} \in \X$ and $\psi \in \partial f(x)$, we have the following inequality: %
    $f(\hat{x}) \geqslant f(x)+\psi(\hat{x}-x)+\frac{\rho}{2}\|\hat{x}-x\|_{\X}^2$.
\end{enumerate}

\begin{proposition} %
\label{prop:compwkcvx}
    Let $f:\X \to [0,B] $ be $\rho$-weakly convex and  $q\geqslant 1$. Then $f(x)^q$ is $qB^{q-1}\rho$-weakly convex.
\end{proposition}
\begin{proof} Note that by Taylor's theorem, for all $y, h \geqslant 0$,
$
(y + h)^q =  y^q +  q\xi ^{q-1}h$
for some $\xi \in [y,y+h]$. 
We have that for all $ x_1, x_2 \in \X$ and $\lambda \in[0,1]$, $0 \leqslant\lambda f\left(x_1\right)+(1-\lambda) f\left(x_2\right) \leqslant B$ and
\begin{align*}
    0 \leqslant f\left(\lambda x_1+(1-\lambda) x_2\right) \leqslant \min\left\{ \lambda f\left(x_1\right)+(1-\lambda) f\left(x_2\right) + \frac{\rho\lambda(1-\lambda)}{2} \left\|x_1-x_2\right\|_{\X}^2, B \right\}. 
    \end{align*}
Hence, since $y \mapsto y^q$ is monotonic on $[0,\infty)$,
    \begin{align*}
    f(\lambda x_1&+(1-\lambda) x_2)^q \\ &\leqslant  \left(\min \left\{\lambda f\left(x_1\right)+(1-\lambda) f\left(x_2\right) + \frac{\rho\lambda(1-\lambda)}{2} \left\|x_1-x_2\right\|_{\X}^2, B\right\}\right)^q\\
    &\leqslant  \left(\lambda f\left(x_1\right)+(1-\lambda) f\left(x_2\right) + \min \left\{\frac{\rho\lambda(1-\lambda)}{2} \left\|x_1-x_2\right\|_{\X}^2, B-\lambda f\left(x_1\right)- (1-\lambda) f\left(x_2\right)\right\}\right)^q\\
    &= \left(\lambda f\left(x_1\right)+(1-\lambda) f\left(x_2\right)\right)^q + q\xi^{q-1} \min \left\{\frac{\rho\lambda(1-\lambda)}{2} \left\|x_1-x_2\right\|_{\X}^2, B-\lambda f\left(x_1\right)- (1-\lambda) f\left(x_2\right)\right\}\\
    &\leqslant \lambda f(x_1)^q + (1-\lambda) f(x_2)^q + q B^{q-1}\frac{\rho\lambda(1-\lambda)}{2} \left\|x_1-x_2\right\|_{\X}^2,
    \end{align*}
    since $\xi \in \left [\lambda f\left(x_1\right)+(1-\lambda) f\left(x_2\right)  , B \right]$ and $y \mapsto y^q$ is convex. 
\end{proof}

For an extended real-valued function $f: \X \rightarrow(-\infty,+\infty]$, let $\operatorname{dom} f:=\{x \in \X: f(x)<+\infty\}$ be its domain and
\[
f^*(\psi):=\sup _{x \in \X}\{\psi(x) -f(x)\}, \qquad \psi \in \X^*,
\]
be its conjugate function, which is always closed, convex, and lower semi-continuous (l.s.c.), see Theorem 4.3 in  \cite{doi:10.1137/1.9781611974997}. 
\begin{theorem}[{Fenchel–Moreau theorem; Theorem 4.2.1 in \cite{borwein2006convex}}]\label{thm:dualdual}
    Let $\X$ be a locally convex Hausdorff space, let $f: \X \to (-\infty,+\infty]$ be convex and l.s.c., and let $\phi: \X \to \X^{**}$ be the canonical embedding which sends $x$ to the pointwise evaluation map $\delta_x:\psi \mapsto \psi(x)$. Then $f = f^{**} \circ \phi $, i.e., for all $x \in \X$
    \[
    f(x) = \max _{\psi \in \X^*}\{\psi(x)-f^*(\psi)\}.
    \]
    If $\X$ is a Hilbert space, then by the Riesz representation theorem it follows that 
     \[
    f(x) = \max _{\x' \in \X}\{\langle x,x'\rangle_{\X} -f^*(x')\}.
    \]
\end{theorem}

\begin{definition}[Moreau envelope and Proximal operator]
For any $f: \X \rightarrow(-\infty,+\infty]$ and $\nu>0$, the \emph{Moreau envelope} and the \emph{proximal mapping} are defined for all $x \in \X$ by
\begin{align*}
f_\nu(x)&:=\inf_{z \in \X} \left\{f(z)+\frac{1}{2 \nu}\|z-x\|_{\X}^2\right\}, \\
\operatorname{prox}_{\nu f}(x)&:=\underset{z \in \X}{\operatorname{argmin}}\left\{f(z)+\frac{1}{2 \nu}\|z-x\|_{\X}^2\right\}.
\end{align*}
\end{definition}
The Moreau envelope of the function has some rather nice properties with relation to the original function:
\begin{theorem}[Moreau envelope and the proximal point map; Lemma 2.5 in \cite{davis2022proximal}]\label{thm:moreauenv_wc}
Consider a $\rho$-weakly convex function $f: \mathbb{R}^d \rightarrow \mathbb{R} \cup\{\infty\}$ and fix a parameter $\mu<\rho^{-1}$. Then, the following are true:
\begin{itemize}
    \item The envelope $f_\mu$ is $C^1$-smooth with its gradient given by
$$
\nabla f_\mu(x)=\mu^{-1}\left(x-\operatorname{prox}_{\mu f}(x)\right) .
$$
 \item The envelope $f_\mu(\cdot)$ is $\mu^{-1}$-smooth and $\frac{\rho}{1-\mu \rho}$-weakly convex meaning:
$$
-\frac{\rho}{2(1-\mu \rho)}\left\|x^{\prime}-x\right\|^2 \leqslant f_\mu\left(x^{\prime}\right)-f_\mu(x)-\left\langle\nabla f_\mu(x), x^{\prime}-x\right\rangle \leqslant \frac{1}{2 \mu}\left\|x^{\prime}-x\right\|^2,
$$
for all $x, x^{\prime} \in \mathbb{R}^d$.
\item The proximal map $\operatorname{prox}_{\mu f}(\cdot)$ is $\frac{1}{1-\mu \rho}$-Lipschitz continuous and the gradient map $\nabla f_\mu$ is Lipschitz continuous with constant $\max \left\{\mu^{-1}, \frac{\rho}{1-\mu \rho}\right\}$.
\end{itemize}
\end{theorem}

\section{Proof of \Cref{thm:bounded}}\label{ap:proof_bounded}

First, we introduce some useful notation. For $\psi \in \X^*$ and $x \in \X$, we define the pairing $\left\langle\psi,x\right\rangle_{\X^*\times \X} := \psi(x)$. 
For $B \subseteq \X^*$ and $x \in \X$, we define $\left\langle B,x\right\rangle_{\X^*\times \X} := \{ \psi(x) : \psi \in B\}.$ 
For sets $A,B$, we write $A \leqslant B$ if for all $a \in A$ and $b \in B$, $a \leqslant b$. We start with a lemma which will also prove useful in other proofs. 

\begin{lemma}\label{lem:sclem}
    Let $f:\X \to (-\infty,+\infty]$ be $\mu$-strongly convex. Then for all $x,z\in \X$
    \[
    \left\langle\partial f(x)-\partial f(z),x-z\right\rangle_{\X^*\times \X
    }\geqslant\mu\|x-z\|_{\X}^2,
    \]
    i.e., for all $\psi \in \partial f(x)$ and $\xi \in \partial f(z)$, 
    \[
\left\langle \psi - \xi ,x-z\right\rangle_{\X^*\times \X}
    \geqslant \mu \|x - z\|^2_{ \X}.
    \]
\end{lemma}
\begin{proof}
    Because $f$ is $\mu$-strongly convex, it follows that for all $x,x'\in \X$ and $\psi \in \partial f(x)$,
    \[
    f(x') \geqslant f(x) + \langle \psi, x'-x\rangle_{\X^* \times \X} + \frac12 \mu \|x'-x\|^2_{\X},
    \]
    and therefore 
    \[
    \langle \psi, x - z \rangle_{\X^*\times \X}  \geqslant f(x) - f(z) + \frac12 \mu \|x - z\|^2_{\X}.
    \]
    By the same argument,
     \[
    \langle \xi, z -x\rangle_{\X^*\times \X}  \geqslant f(z) - f(x) + \frac12 \mu \|x - z\|^2_{\X},
    \]
    and the result follows. 
\end{proof}
\begin{lemma}
    \label{lem:Lipsubdiff}
    Let $f:\X \to (-\infty,+\infty]$ be weakly convex and $L$-Lipschitz. Then for all $x,z \in \X$ and $\psi \in \partial f(x)$, $\langle \psi, z-x\rangle_{\X^*\times \X}  \leqslant L \|x - z\|_{\X}$. 
\end{lemma}
\begin{proof}
    Let $x'_t:= tz + (1-t)x$. Then as $t \to 0$, $\x'_t \to x$ and so 
    \[
    f(x'_t) \geqslant f(x) + \langle \psi, x'_t-x\rangle_{\X^*\times\X}  + o(\|x'_t - x\|_{\X})
    \]
    and therefore, since $x'_t - x = t(z-x)$, 
    \[
    t \langle \psi, z-x\rangle_{\X^*\times\X}  + o(t) \leqslant f(x'_t) - f(x) \leqslant L\|x'_t - x\|_{\X} = tL\|z-x\|_{\X}.
    \]
    Dividing both sides by $t$ and taking $t \to 0$ completes the proof.
\end{proof}

\begin{proof}[Proof of \Cref{thm:bounded}]
    For any $x,z\in\X$, by \Cref{lem:sclem}, $$\left\langle\partial R_{sc}(x)-\partial R_{sc}(z),x-z\right\rangle_{%
    \X} \geqslant\mu\|x-z\|_{\X}^2.$$ 
Therefore, in the Lipschitz case, for $\R=R_{wc}+R_{sc}$, we have: 
    \begin{align*}
    \left\langle\left(\partial R_{wc}(x) + \partial R_{sc}(x)\right),z-x\right\rangle_{
        \X^*\times 
        \X} 
        &\subseteq -\left\langle\partial R_{sc}(x)-\partial R_{sc}(z),x-z\right\rangle_{\X^*\times 
        \X} + \left\langle\partial R_{wc}(x)+\partial R_{sc}(z),z-x\right\rangle_{
        \X^*\times 
        \X}\\
        &\leqslant -\mu\|x-z\|_{\X}^2 + \left(L_R+\left\|\partial R_{sc}(z)\right\|\right)\|x-z\|_{\X},
    \end{align*}
where the first half of the inequality follows from the above, and the second half since, for all $\psi \in \partial R_{wc}(x)$ and $\xi \in \partial R_{sc}(z)$ 
\begin{align*}
 \langle \psi, z-x\rangle_{\X^*\times\X}  & \leqslant L_{R}\|x-z\|_{\X} \quad \text{by \Cref{lem:Lipsubdiff}, and}\\
 \langle \xi, z-x\rangle_{\X^*\times\X}  &\leqslant \|\xi\|_{\X^*}\|x-z\|_{\X} \leqslant \left\|\partial R_{sc}(z)\right\|\|x-z\|_{\X}.
\end{align*}
 Then, for any $\hat{x}$ a stationary point we have that $0\in \partial R_{wc}(\hat x) + \partial R_{sc}(\hat x) $ and so
\[
0 \leqslant -\mu \|x-z\|_{\X}^2 + \left(L_R+\left\|\partial R_{sc}(z)\right\|\right)\|x-z\|_{\X}
\]

which gives
\begin{equation*}
 \|\hat{x}-z\|_{\X}\leqslant \frac{1}{\mu }\left(L_R+\|\partial R_{sc}(z)\|\right) .
\end{equation*}
In the bounded weak convex case, using weak convexity instead: 
\[R_{wc}(z)\geqslant R_{wc}(x)+\left\langle\partial R_{wc}(x),z-x\right\rangle_{
\X^* \times 
\X}-\frac{\gamma}{2}\|z-x\|_{\X}^2.\]
Then we get the following:
    \begin{align*}
\left\langle\left(\partial{R}_{wc}(x) + \partial{R}_{sc}(x)\right),z-x\right\rangle_{
\X^* \times 
\X}
        &\subseteq -\left\langle\partial R_{sc}(x)-\partial R_{sc}(z),x-z\right\rangle_{\X^*\times 
        \X} + \left\langle\partial R_{wc}(x)+\partial R_{sc}(z),z-x\right\rangle_{\X^*\times 
        \X}\\
        &\leqslant -\mu\|x-z\|_{\X}^2 + \frac{\gamma}{2}{\|x-z\|_{\X}^2}+\left(\|\partial R_{sc}(z)\|\right)\|x-z\|_{ \X} +  R_{wc}(z).
    \end{align*}
 Taking $\hat{x}$ to be any stationary point and completing the square, this implies that
\begin{equation*}
 \left(\|\hat{x}-z\|_{\X}- \frac{1}{2(\mu-\frac{\gamma}{2})}\left(\|\partial R_{sc}(z)\|\right)\right)^2\leqslant  \frac{1}{4(\mu-\frac{\gamma}{2})^2}\left(\|\partial R_{sc}(z)\|\right)^2 + \frac{1}{(\mu-\frac{\gamma}{2})}R_{wc}(z),
\end{equation*}
or rearranging and using the triangle inequality: 
\begin{equation*}
 \|\hat{x}-z\|_{\X} \leqslant \frac{1}{(\mu-\frac{\gamma}{2})}\|\partial R_{sc}(z)\| + \sqrt{\frac{1}{(\mu-\frac{\gamma}{2})}R_{wc}(z)}.
\end{equation*}
    \end{proof}

\section{Unbounded critical points}\label{ap:counterexample}
Example construction of how to get infinitely many minimisers for the weak convex case.
\begin{nb}
        It may not be immediately clear why the condition is $\gamma<2\mu$ and not simply $\gamma<\mu$, thus requiring strong convexity of the overall functions. The idea for why arises from the lower boundedness of the weak convex function. As an example, one can consider the 1D case. The overall value of the function can not decrease arbitrarily, thus the derivative cannot be strictly negative. And thus heuristically, as the gradient of a strong convex function simply increases, the weak convex one has to decrease and then increase, overall resulting in the factor of 2, since it has to go up and then back down. \\
    Based on this heuristic understanding we can answer the question of whether it is necessary to have $\gamma<2\mu$ to have bounded stationary points, or if it can be made larger. As it turns out, we can construct functions with $\gamma>2\mu$ with unbounded stationary points. 
    Assuming we are working in 1D on a positive real line, and letting $R_{sc}=\frac12x^2$, assume we start with a stationary point at $x=x_0=1$, with value $R_{wc}=0$, i.e. derivative intersects $y=-x$. Now, we require that the function is bounded from below by 0 and thus gradient has to now become positive in such a way to increase and decrease quickly enough to once again cross $y=-x$. The largest area is found by taking the gradient to be $-\gamma x$. With this, we can find that the next intersection, and thus stationary point would be at $x_1=\frac{\gamma}{\gamma-2}x_0$. And continuing in such an iterative way, we can construct a positive weakly convex function that has unboundedly many stationary points. For concreteness, we can write the following exact form for the gradient of the function, if we let $m=\frac{\gamma}{\gamma-2}$:
    \begin{align*}%
        R'_{wc} = \left(m+\gamma\right)m^n-\gamma x,\quad\text{for } x\in[m^{n},m^{n+1}), n\in\mathbb{Z}
    \end{align*}
    \Cref{fig:wc_example} illustrates what the resulting functions looks like, and in particular, we can see that for any $\gamma>2$, we have infinitely many stationary points. 
    \begin{figure}[h]
        \centering
        \includegraphics[width=.4\textwidth]{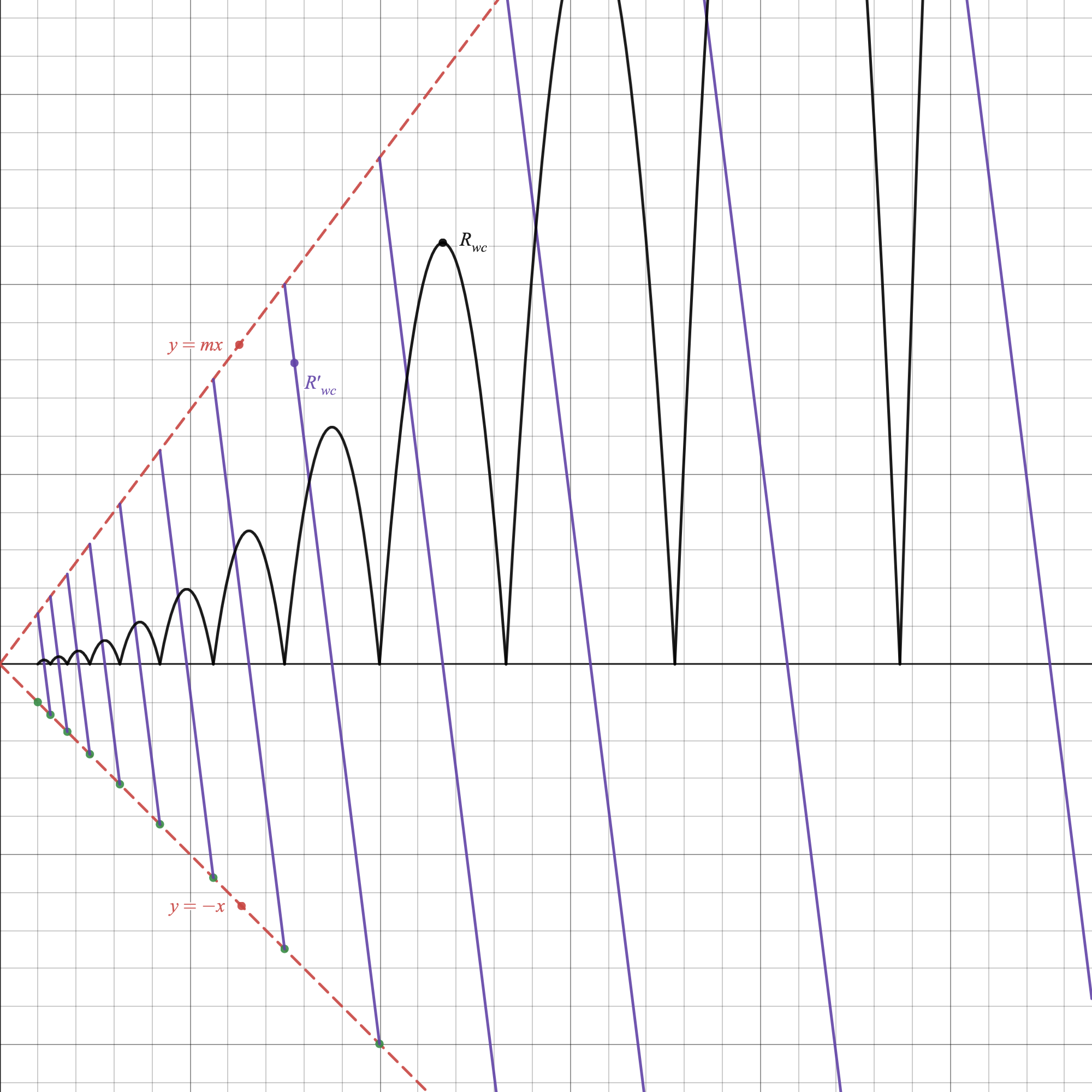}
        \caption{Example of a positive weakly convex function with $\gamma > 2\mu$ and unbounded stationary points.}
        \label{fig:wc_example}
    \end{figure}
\end{nb}

\section{Proofs and a counterexample for \Cref{sec:InvP}}
\subsection{Proof of \Cref{thm:stability}}\label{ap:stabilityproof}
\begin{proof}
One of the main limitations in achieving provable stability and convergence guarantees lies in the ability of proving boundedness of iterates $x_k$. However, thanks to the form of the regulariser, this is achieved almost immediately.
By \Cref{lem:sclem}, for all $u \in \X$, 
\[
\left\langle\partial R_{sc}(x_k)-\partial R_{sc}(u),x_k-u\right\rangle_{\X^*\times 
\X}\geqslant\mu\|x_k-u\|_{\X}^2
\]
 and for $\D$ convex: for all $x,z \in \X$ and $y \in \Y$, 
\[ \D(\A z,y)\geqslant\D(\A z,y)-\D(\A x,y)\geqslant \langle\partial_x\D(\A x,y),z-x\rangle_{\X^*\times
\X}\]
For any critical point $x_k$ and $z\in\X$, we have: 
\begin{align*}
0&\in\left\langle\partial_x\D(\A x_k,y_k)+\alpha\left(\partial\mathcal{R}_{wc}(x_k) + \partial\mathcal{R}_{sc}(x_k)\right),z-x_k\right\rangle_{\X^*\times
\X}\\
&\subseteq \langle\partial_x\D(\A x_k,y_k),z-x\rangle_{\X^*\times
\X} - \alpha\left\langle\partial\R_{sc}(x_k)-\partial\R_{sc}(z),x_k-z\right\rangle_{\X^*\times
\X} + \alpha\left\langle\partial\R_{wc}(x_k)+\partial\R_{sc}(z),z-x_k\right\rangle_{\X^*\times
\X}\\
        &\leqslant\D(\A z,y_k) - \alpha\left\langle\partial\R_{sc}(x_k)-\partial\R_{sc}(z),x_k-z\right\rangle_{%
        \times\X} + \alpha\left\langle\partial\R_{wc}(x_k)+\partial\R_{sc}(z),z-x_k\right\rangle_{%
        \X}.
\end{align*}
    
Therefore, since $\gamma<2\mu$, by \Cref{thm:bounded} and by rearranging and completing the square, we get:
\begin{equation}\label{eq:gammamubdd}
 \left(\|x_k-z\|_{\X}- \frac{1}{2(\mu-\frac{\gamma}{2})}\left(\|\partial\R_{sc}(z)\|\right)\right)^2\leqslant \frac{1}{\alpha(\mu-\frac{\gamma}{2})}\D(\A z,y_k) + \frac{1}{4(\mu-\frac{\gamma}{2})^2}\left(\|\partial\R_{sc}(z)\|\right)^2 + \frac{1}{(\mu-\frac{\gamma}{2})}\R_{wc}(z),
\end{equation}
which together with convergence of $y_k$ and continuity of $\D$ implies that the $(x_k)$ are bounded, and hence have a weak 
convergent subsequence since $\X$ is reflexive, via the Banach--Alaoglu theorem. All that remains is to show that all cluster points $x_+$ must be critical points. Passing to a subsequence, let $x_k \rightharpoonup x_{+}$. 
Then in the case of $\gamma \leqslant \mu$:
\begin{align*}
    \D(\A x_+,y^{\delta})+\alpha\mathcal{R}(x_+) &\leqslant \liminf _{k} \D(\A x_k,y_k) + \alpha\mathcal{R}(x_k) \quad\text{since $\R$ and $\D$ are weakly sequentially l.s.c.}\\
        &\leqslant \liminf _{k} \D(\A u,y_k) + \alpha\mathcal{R}(u) + \frac{\alpha}{2}(\gamma-\mu)\|x_k-u\|_{\X}^2 \quad\text{$\forall u\in\X$, by weak convexity}\\
        &\leqslant  \D(\A u,y^{\delta}) + \alpha\mathcal{R}(u) + \frac{\alpha}{2}(\gamma-\mu)\|x_{+}-u\|_{\X}^2 \quad\text{ $\forall u\in \X$,},
    \end{align*}
    where the last line follows since $x_k \rightharpoonup x_{+}$ and $\|\cdot\|_{\X}$ is weakly sequentially l.s.c., and therefore $0\in\partial\mathcal{J}_{\alpha, y^\delta}(x_{+})$ by definition of weak convexity. 
    
    In the case of $\mu < \gamma < 2\mu$, $R_{sc} - \frac12 \mu \|\cdot\|_{\X}^2$ weakly sequentially l.s.c., 
    \begin{align*}
    &\D(\A x_+,y^{\delta})+\alpha\mathcal{R}(x_+) - \frac{\alpha}{2}(\gamma-\mu)\|x_{+}-u\|_{\X}^2 
    \\& = \D(\A x_+,y^{\delta})+\alpha\mathcal{R}(x_+) - \frac{\alpha}{2}(\gamma-\mu)\|x_{+}\|_{\X}^2 + \alpha(\gamma-\mu)\langle x_{+},u\rangle_{\X} - \frac{\alpha}{2}(\gamma-\mu)\|u\|_{\X}^2 
    \\&= \D(\A x_+,y^{\delta})+\alpha\left(\mathcal{R}(x_+)- \frac{\mu}{2}\|x_{+}\|_{\X}^2 \right ) + \frac{\alpha}{2}(2\mu-\gamma)\|x_{+}\|_{\X}^2 + \alpha(\gamma-\mu)\langle x_{+},u\rangle_{\X} - \frac{\alpha}{2}(\gamma-\mu)\|u\|_{\X}^2 
    \\ &\leqslant \liminf _{k} \D(\A x_k,y_k) + \alpha\left(\mathcal{R}(x_k)- \frac{\mu}{2}\|x_{k}\|_{\X}^2 \right ) + \frac{\alpha}{2}(2\mu-\gamma)\|x_{k}\|_{\X}^2 + \alpha(\gamma-\mu)\langle x_{k},u\rangle_{\X} - \frac{\alpha}{2}(\gamma-\mu)\|u\|_{\X}^2\\ &\qquad \text{since $\R-\frac{\mu}{2}\|\cdot\|_{\X}^2$, $\|\cdot\|_{\X}$,  and $\D$ are weakly sequentially l.s.c., and inner product is weakly continuous}
    \\ &= \liminf _{k} \D(\A x_k,y_k) + \alpha\mathcal{R}(x_k) - \frac{\alpha}{2}(\gamma-\mu)\|x_k-u\|_{\X}^2 
        \\ &\leqslant \liminf _{k} \D(\A u,y_k) + \alpha\mathcal{R}(u)  \quad\text{$\forall u\in\X$, by weak convexity  and the criticality of $x_k$}\\
        &= \D(\A u,y^{\delta}) + \alpha\mathcal{R}(u) \quad\text{ $\forall u\in \X$, since $y_k \to y^\delta$ and $\D$ is continuous in its second argument}.
    \end{align*}
\end{proof}
\subsection{Proof of \Cref{thm:convreg}}\label{ap:convregproof}
\begin{proof}
    Similar to the theorem above, we first want to show boundedness of iterates $x_k$. To achieve this, we use the fact that $y^0 = \A u$, and therefore by \Cref{assumptions}(5)
    \[
    \D(\A u,y^{\delta_k})  = \D(y^0,y^{\delta_k})
    \leqslant C\left( \D(y^0,y^0) + \|y^{\delta_k}-y^0\|^p \right) =  
    C\delta_k^p.
    \]
    Hence, 
by \Cref{eq:gammamubdd}:
\begin{equation*}
 \left(\|x_k-u\|_{\X}- \frac{1}{2(\mu-\frac{\gamma}{2})}\left(\|\partial\R_{sc}(u)\|\right)\right)^2\leqslant \frac{C}{\alpha_k(\mu-\frac{\gamma}{2})}\delta_k^p  + \frac{1}{4(\mu-\frac{\gamma}{2})^2}\left(\|\partial\R_{sc}(u)\|\right)^2 + \frac{1}{(\mu-\frac{\gamma}{2})}\R_{wc}(u),
\end{equation*}
Therefore, as in the previous theorem, $(x_k)$ weakly
has convergent subsequences (due to reflexivity and the Banach--Alaoglu theorem) 
since $\delta_k^p/\alpha_k$ is assumed to converge. Passing to such a subsequence, let $x_k \rightharpoonup x_{+}$. We wish to show that $x_{+}$ is $\R$-criticising. We first show that $\A x_{+} = y^0$:
\begin{align*}
       0 \leqslant  \D(\A x_+,y^0) 
        &\leqslant \liminf _{k} \D(\A x_k,y^{\delta_k}) \quad \text{ as $\D$ is weakly sequentially l.s.c.}
        \\&\leqslant \liminf _{k} \D(\A x_k,y^{\delta_k}) + \alpha_k\mathcal{R}(x_k)  \quad \text{ as $\alpha_k\mathcal{R}(x_k)\geqslant 0$ }
        \\&\leqslant \liminf _k \D(\A u,y^{\delta_k}) + \alpha_k\mathcal{R}(u) + \frac{1}{2}\alpha_k(\gamma-\mu)\|x_k-u\|_{\X}^2 \quad\text{since $x_k \in \operatorname*{crit} \mathcal{J}_{\alpha_k,y^{\delta_k}}$, where $\A u = y^0$}\\
        &\leqslant \liminf_k  C\delta_k^p + \alpha_k\mathcal{R}(u) + \frac{1}{2}\alpha_k(\gamma-\mu)\|x_k-u\|_{\X}^2 \quad\text{ by \Cref{assumptions}(5)}\\
        &=0,
    \end{align*}
hence $\A x_{+} = y^0$. It remains to prove that, for all $u\in \X$ such that $\A u = y^0$, 
\[
\R(u) \geqslant \R(x_+) + \frac12(\mu - \gamma)\|u - x_{+}\|_{\X}^2. 
\]
By the criticality of $x_k$, the non-negativity of $\D$, and \Cref{assumptions}(5):
\[
\alpha_k\mathcal{R}(x_k) \leqslant  \D(\A x_k,y^{\delta_k}) + \alpha_k\mathcal{R}(x_k) \leqslant C\delta_k^p + \alpha_k\mathcal{R}(u) + \frac{1}{2}\alpha_k(\gamma-\mu)\|x_k-u\|_{\X}^2,
\]
and so 
\begin{equation}
    \label{eq:Rineq}
    \mathcal{R}(x_k) \leqslant C\frac{\delta_k^p}{\alpha_k} + \mathcal{R}(u) + \frac{1}{2}(\gamma-\mu)\|x_k-u\|_{\X}^2.
\end{equation}
Therefore, if $\gamma \leqslant \mu$, 
\begin{align*}
\R(x_+) &\leqslant \liminf_k \R(x_k)  \quad\text{as $\R$ is weakly sequentially l.s.c.} \\
&\leqslant \liminf_k C\frac{\delta_k^p}{\alpha_k} + \mathcal{R}(u) + \frac{1}{2}(\gamma-\mu)\|x_k-u\|_{\X}^2 \quad \text{by \cref{eq:Rineq}}\\
&\leqslant  \mathcal{R}(u) + \frac{1}{2}(\gamma-\mu)\|x_+-u\|_{\X}^2, \quad\text{ since $\|\cdot\|_{\X}$ is weakly sequentially l.s.c.,}
\end{align*}
as desired.
In the case $\mu < \gamma < 2\mu$, $R_{sc} - \frac12 \mu \|\cdot\|_{\X}^2$ is weakly sequentially l.s.c., the proof runs:
    \begin{align*}
\R(x_+) - \frac12 \mu \|x_+-u\|_{\X}^2
& = \R(x_+) - \frac12 \mu \|x_+\| + \mu\langle x_+, u\rangle- \frac12 \mu\|u\|_{\X}^2 \\
&\leqslant \liminf_k \R(x_k) - \frac12 \mu \|x_k \|_{\X}^2 + \mu\langle x_k, u\rangle- \frac12 \mu\|u\|_{\X}^2\\
&= \liminf_k \R(x_k) - \frac12 \mu \|x_k - u \|_{\X}^2 \\
&\leqslant \liminf_k C\frac{\delta_k^p}{\alpha_k} + \mathcal{R}(u) + \frac{1}{2}(\gamma-2\mu)\|x_k-u\|_{\X}^2\\
&\leqslant  \mathcal{R}(u) + \frac{1}{2}(\gamma-2\mu)\|x_+-u\|_{\X}^2.
\end{align*}

Finally, whenever the solution to $\A u = y^0$ is unique, then by the above every subsequence of $x_k$ has a subsubsequence converging weakly to that $u$. It follows that $x_k$  converges  weakly to $u$, as if $x_k$ did not converge weakly to $u$, then there would be a neighbourhood $\mathcal{N}$ of $u$ (in the weak topology) and a subsequence $x'_\ell$ of $x_k$ such that $x'_\ell \notin \mathcal{N}$ for all $\ell$. This subsequence would have no subsubsequence converging weakly to $u$, contradicting the above. 
\end{proof}

\subsection{Counterexample to \Cref{thm:stability,thm:convreg} in the general weakly convex setting}\label{ap:infDcounterexample}

\begin{example}\label{ex:ctrexample}
Let: \begin{itemize}
    \item $\X:= \ell^2(\mathbb{N})$, i.e. $\X := \{ (a_n)_{n \in \mathbb{N}}: \sum_{n=1}^{\infty} a_n^2 < \infty \}$ with $\langle (a_n), (b_n) \rangle_{\X}:= \sum_{n=1}^\infty a_n b_n$,
    \item $\Y = \mathbb{R}$,
    \item $\A = 0$,
    \item $\D(y,z):= (y-z)^2$ (which can be checked to satisfy \Cref{assumptions}(4-5)) and so $\D(\A x, y) = y^2 $,
    \item $y_k = y^\delta = y^{\delta_k}= y^0 = 0$, and 
    \item $\R(x) := \begin{cases}
        f(x), &\|x\|_{\X} \leqslant \frac12,\\
        \left(\|x\|^2_{\X} - 1 \right)^2, &\|x\|_{\X} > \frac12,
    \end{cases}$ where $f$ is any smooth weakly convex function for which $0$ is not a critical point and for which every derivative (including the $0^\text{th}$) of $f$ at any $x$ with $\|x\|_{\X} = \frac12$ agrees with that of 
    $\left(\|x\|^2_{\X} - 1 \right)^2$. 
    \end{itemize}

It follows that $\R$ is weakly convex because there exists $\rho \geqslant 0$ such that $f+ \rho\|x\|_{\X}^2$ is convex, and $ \left(\|x\|^2_{\X} - 1 \right)^2 + 2\|x\|_{\X}^2 = \|x\|^4_{\X} + 1$ is convex since it is the composition of $t \mapsto t^4 + 1$ (which is a convex and monotone on $[0,\infty)$) with $\|\cdot\|_{\X}$ which is convex and has range $[0,\infty)$. Hence $\R(x) + \max\{\rho,2\}\|x\|_{\X}^2$ is convex. 
    
    Then for all $\alpha > 0$ and $y \in \Y$, $\partial \mathcal{J}_{\alpha,y}(x) = \alpha \partial\R(x)$. Hence, for all $\alpha >0$, $x_k \in \operatorname*{crit}_{x \in \X} \mathcal{J}_{\alpha,y_k}$ if and only if  $x_k \in \operatorname*{crit}_{x \in \X} \R$. But by construction $\{\|x\|_{\X}=1\} \subseteq \operatorname*{crit}_{x \in \X} \R$ and $0 \notin \operatorname*{crit}_{x \in \X} \R$. Therefore, we can choose $x_k = e_k$, i.e. the sequence which is 1 when $n =k$ and $0$ otherwise. Then it is a well-known fact that $e_k \rightharpoonup 0$, which by construction is not a critical point of $\mathcal{J}_{\alpha,y^\delta}$. Hence \Cref{thm:stability} does not hold. 

    For \Cref{thm:convreg}, every $x \in \X$ has $\A x = y^0$. It follows that $x_+$ is an $\R$-criticising solution if and only if $x_+$ is a critical point of $\R$. We can once again take $x_k = e_k \in \operatorname*{crit}_{x \in \X} \mathcal{J}_{\alpha_k,y^{\delta_k}}$, with weak limit 0, which is not a critical point of $\R$. Hence, \Cref{ex:ctrexample} is also a counterexample to \Cref{thm:convreg} in the general setting of $\R$ weak convex.

\begin{nb}
    The problematic critical points here can be chosen to indeed be global minimisers of $\R$, so this issue is not caused by the fact that we are in the setting of critical points.
\end{nb}
    \begin{nb}
         We could modify $\R$ to also be globally Lipschitz by setting $\R(x) = g(x)$ for $\|x\|_{\X} > 2$ where $g$ is any smooth, globally Lipschitz, weakly convex function whose derivatives all agree with those of $\left(\|x\|^2_{\X} - 1 \right)^2$ at $\|x\|_{\X} = 2$. 
    \end{nb}
   \begin{nb}
       This counterexample does not work in the finite-dimensional setting, as in that case the weak and norm topologies coincide, and hence $\{\|x\|_{\X}=1\}$ contains all its limit points, and so does not have 0 as a limit point.
   \end{nb}
\end{example}

\section{Proofs and definitions for \Cref{sec:Opt}}

\subsection{Proof of \Cref{thm:descent}}\label{ap:descent}
\begin{proof}
Written this way, we have by weak convexity-concavity of $L(x, y)$, assuming that $F^*$ is $\scv{\mu}$-strongly convex, that
\begin{align}
L\left(x^{k+1}, y\right)-L\left(x, y^{k+1}\right)
&=L\left(x^{k+1},y\right)-L\left(x^{k+1}, y^{k+1}\right)+L\left(x^{k+1}, y^{k+1}\right)-L\left(x, y^{k+1}\right) \nonumber\\
&\leqslant \left\langle Tz^{k+1}, z^{k+1}-z\right\rangle_{\X\times\Y} + \frac{\wcv{\rho}}{2}\|x^{k+1}-x\|_{\X}^2 - \frac{\scv{\mu}}{2}\|y^{k+1}-y\|_{\Y}^2 \nonumber\\
&=\left\langle z^k-z^{k+1}, z^{k+1}-z\right\rangle_M  + \frac{\wcv{\rho}}{2}\|x^{k+1}-x\|_{\X}^2 - \frac{\scv{\mu}}{2}\|y^{k+1}-y\|_{\Y}^2     \label{eq:k_gap}\\
&= \frac{1}{2}\left\|z^k-z\right\|_M^2-\frac{1}{2}\left\|z^{k+1}-z\right\|_M^2-\frac{1}{2}\left\|z^k-z^{k+1}\right\|_M^2 
+ \frac{\wcv{\rho}}{2}\|x^{k+1}-x\|_{\X}^2 - \frac{\scv{\mu}}{2}\|y^{k+1}-y\|_{\Y}^2. \nonumber%
\end{align}
But then adding together the inequality above for $k$ and $k+1$ evaluated at $z=(x^{k},y^{k+1})$ 
\begin{align*}
    L\left(x^{k+1}, y^{k+1}\right)-L&\left(x^{k}, y^{k}\right)\\
    &= L\left(x^{k+1},y^{k+1}\right)-L\left(x^{k}, y^{k+1}\right)+L\left(x^{k}, y^{k+1}\right)-L\left(x^{k}, y^{k}\right) \\   
    &\leqslant \frac{1}{2}\left\|y^k-y^{k+1}\right\|_M^2-\frac{1}{2}\left\|x^{k+1}-x^k\right\|_M^2-\frac{1}{2}\left\|z^k-z^{k+1}\right\|_M^2 + \\ &\quad-\frac{1}{2}\left\|z^{k-1}-z^{k}\right\|_M^2 + \frac{\wcv{\rho}}{2}\|x^{k+1}-x^k\|_{\X}^2 -\frac{\scv{\mu}}{2}\|y^{k+1}-y^k\|_{\Y}^2 + \frac{1}{2}\left\|\left(\begin{array}{c}x^{k-1}-x^{k} \\ y^{k-1}-y^{k+1}\end{array}\right)\right\|_M^2\\
    &= \left\|y^k-y^{k+1}\right\|_M^2-\frac{1}{2}\left\|x^{k}-x^{k+1}\right\|_M^2-\frac{1}{2}\left\|z^k-z^{k+1}\right\|_M^2 + \frac{\wcv{\rho}}{2}\|x^{k}-x^{k+1}\|_{\X}^2 \\&\quad-\frac{\scv{\mu}}{2}\|y^{k+1}-y^k\|_{\Y}^2 +\left\langle\left(\begin{array}{c}x^{k-1}-x^{k} \\ y^{k-1}-y^{k}\end{array}\right),\left(\begin{array}{c}0 \\ y^{k}-y^{k+1}\end{array}\right)\right\rangle_M %
    \\
     &\leqslant \frac{3}{2}\left\|y^k-y^{k+1}\right\|_M^2-\frac{1}{2}\left\|x^{k+1}-x^k\right\|_M^2-\frac{1}{2}\left\|z^k-z^{k+1}\right\|_M^2 + 
     \\ &\quad
     +\frac{1}{2}\left\|z^{k-1}-z^{k}\right\|_M^2 + \frac{\wcv{\rho}}{2}\|x^{k+1}-x^k\|_{\X}^2 - \frac{\scv{\mu}}{2}\|y^{k+1}-y^k\|_{\Y}^2.
\end{align*}
We can write this illustrating descent:
\begin{align*}
    L\left(x^{k}, y^{k}\right) +\frac{1}{2}\left\|z^{k-1}-z^{k}\right\|_M^2
    \geqslant %
    L\left(x^{k+1}, y^{k+1}\right)&+
    \frac{1}{2}  \left\|z^k-z^{k+1}\right\|_M^2  
    \\
    &+\left[\frac{1}{2}(\scv{\mu}\sigma-3)\left\|y^{k}-y^{k+1}\right\|_{M}^2 +\frac{1}{2}(1-\wcv{\rho}\tau)\left\|x^{k+1}-x^k\right\|_{M}^2 \right].
\end{align*}
This becomes descent for the following restrictions on the parameters:
\begin{equation*}
\begin{cases}
    \tau\sigma<\frac{1}{\|\A\|^2},\\
    \tau\wcv{\rho}<1,\\
    \mu\scv{\sigma}>3.
\end{cases}    
\end{equation*}
\end{proof}
\subsection{Proof of \Cref{thm:pd_results}}\label{as:pd_results}
\begin{proof}
In analogy with \cite{bolte2014proximal}, we shall show require two results - one on descent, as in \Cref{thm:descent}, and one on subgradient boundedness, which arises from the following:
\begin{align*}
\left\|\partial \mathcal{L}(z^{k+1},z^{k})\right\| &= \left\|\left(\begin{array}{c}\left(\begin{array}{c} \partial_x L(z^{k+1}) \\ \partial_y L(z^{k+1}) \end{array}\right) + M[z^{k+1}-z^k] \\ -M[z^{k+1}-z^k]\end{array}\right)\right\| \\
&\leqslant \left\|\left(\begin{array}{c}\left(\begin{array}{c} \partial_x L(z^{k+1}) \\ \partial_y L(z^{k+1}) \end{array}\right)  \\ 0\end{array}\right)\right\| + \left\|\left(\begin{array}{c} M[z^{k+1}-z^k] \\ -M[z^{k+1}-z^k]\end{array}\right)\right\| \\ &= 3\|M(z^{k+1}-z^k)\|   .
\end{align*}
Overall, the proof is similar to that of \cite{guo2023preconditioned} and \cite{bolte2014proximal}, and only the first part will be shown here, for compactness. 
Now, from \Cref{thm:descent}:
    \begin{equation*}
      \sum_{k=1}^{K}\left[\frac{1}{2}(\scv{\mu}\sigma-3)\left\|y^{k}-y^{k+1}\right\|_{M}^2 +\frac{1}{2}(1-\wcv{\rho}\tau)\left\|x^{k+1}-x^k\right\|_{M}^2 \right] \leqslant \mathcal{L}(z^1,z^0) - \mathcal{L}(z^{K+1},z^{K}).
    \end{equation*}
But by assumptions above $\inf _k L(x_k,y_k) > -\infty$ and by boundedness $\mathcal{L}^k > -\infty$, we have $\mathcal{L}^k$ is non-increasing and thus converges to $\mathcal{L}^*$. Thus, taking $K\to\infty$: 
\begin{equation*}
    \sum_{k=1}^{\infty}\left[\frac{1}{2}(\scv{\mu}\sigma-3)\left\|y^{k}-y^{k+1}\right\|_{M}^2 +\frac{1}{2}(1-\wcv{\rho}\tau)\left\|x^{k+1}-x^k\right\|_{M}^2 \right] \leqslant \mathcal{L}^0 - \mathcal{L}^*.
    \end{equation*}
This, in particular, implies that the series is square-summable and furthermore that 
\[
\lim _{k \rightarrow \infty}\left\|x^{k+1}-x^k\right\|_{\X}=0 \text { and } \lim _{k \rightarrow \infty}\left\|y^{k+1}-y^k\right\|_{\Y}=0.
\]
Recalling that $\operatorname{dist}\left(0, \partial L(z_k)\right) = \operatorname{dist}\left(0, T(z_k)\right) = \left\|z^{k+1}-z^{k}\right\|_M \leqslant \left\|y^{k+1}-y^{k}\right\|_M + \left\|x^{k+1}-x^{k}\right\|_M\to 0$ as $k\to\infty$, and for $\nu = \min\{\scv{\mu}\sigma-3,1-\wcv{\rho}\tau\}$:
\begin{equation*}
    \frac{\nu}{4}\left\|z^{k+1}-z^{k}\right\|^2_M\leqslant\left[\frac{1}{2}(\scv{\mu}\sigma-3)\left\|y^{k}-y^{k+1}\right\|_{M}^2 +\frac{1}{2}(1-\wcv{\rho}\tau)\left\|x^{k+1}-x^k\right\|_{M}^2 \right].
\end{equation*}
But also from above we find that 
\begin{equation*}
    \min _{k} \left[\frac{1}{2}(\scv{\mu}\sigma-3)\left\|y^{k}-y^{k+1}\right\|_{M}^2 +\frac{1}{2}(1-\wcv{\rho}\tau)\left\|x^{k+1}-x^k\right\|_{M}^2 \right] \leqslant \frac{1}{K}\left(\mathcal{L}(z^1,z^0) - \mathcal{L}(z^{K+1},z^{K})\right).
\end{equation*}
Which we can combine to find: 
\[
\min _{k} \operatorname{dist}\left(0, \partial L(z^k)\right) \leqslant \frac{2}{(\nu K)^{1 / 2}} \sqrt{\mathcal{L}(z^1,z^0) - \mathcal{L}(z^{K+1},z^{K})}.
\]
\end{proof}

\subsection{The Kurdyka--\L ojasiewicz inequality}
\label{ap:KL}

In \cite{kurdyka1998gradients} Kurdyka provided a generalisation of the \L ojasiewicz inequality, with extensions to the nonsmooth setting in \cite{bolte2007lojasiewicz}. %
\begin{definition}
Let $\eta \in(0,+\infty]$. Denote with $\Phi_\eta$ the set of all concave and continuous functions $\varphi:[0, \eta) \rightarrow[0,+\infty)$ which satisfy: $\varphi(0)=0$; $\varphi$ is $\mathcal{C}^1$ on $(0, \eta)$ and continuous at 0; and for all $s \in(0, \eta)$, $\varphi^{\prime}(s)>0$.
\end{definition}
\begin{definition}
Let $\Psi: \mathbb{R}^d \rightarrow \mathbb{R} \cup\{+\infty\}$ be proper and lower semicontinuous (l.s.c.). Then $\Psi$ is said to have the \textit{Kurdyka--Łojasiewicz (K\L) property} at a point $\hat{v} \in \operatorname{dom} \partial \Psi:=$ $\left\{v \in \mathbb{R}^d: \partial \Psi(v) \neq \varnothing\right\}$, if there exists $\eta \in(0,+\infty]$, a neighborhood $V$ of $\hat{v}$, and a function $\varphi \in \Phi_\eta$ such that for any
$$
v \in V \cap\left\{v \in \X: \Psi(\widehat{v})<\Psi(v)<\Psi(\widehat{v})+\eta\right\}
$$
the Kurdyka--Łojasiewicz inequality holds
$$
\varphi^{\prime}(\Psi(v)-\Psi(\widehat{v})) \cdot \operatorname{dist}(\mathbf{0}, \partial \Psi(v)) \geqslant 1 .
$$
If $\varphi(s)=c s^{1-\theta}$ for $c > 0$ and $\theta \in [0,1)$, then $\theta$ is called the K{\L} exponent of $\Psi$ at $\hat v$. If $\Psi$ has the KŁ property at each point of $\operatorname{dom} \partial \Psi$, then $\Psi$ is called a KŁ function (or just, K{\L}).  
\end{definition}
Examples of K{\L} functions include semialgebraic, subanalytic, uniformly convex functions (see \cite{Attouch2010} and the references therein) and all typical neural networks.
\begin{theorem}\label{thm:NNKL}
    Let $\Psi:\mathbb{R}^d \rightarrow \mathbb{R}$ be a deep neural network with every activation function a continuous piecewise analytic function with finitely many pieces (e.g., \texttt{ReLU}, \texttt{sigmoid}). Then $\Psi$ is a K{\L} function and for all $\hat v \in \operatorname{dom}\partial\Psi$, there exists $\theta \in [0,1)$ such that $\Psi$ has K{\L} exponent $\theta$ at $\hat v$.
\end{theorem}
\begin{proof}
    The network $\Psi$ is a finite composition of continuous piecewise analytic functions with finitely many pieces, and hence is a continuous piecewise analytic function with finitely many pieces. It follows that $\Psi$ is subanalytic and so the result follows by Theorem 3.1 of \cite{bolte2007lojasiewicz}.
\end{proof}

\subsection{Proof of \Cref{thm:ergodic_rate}}\label{ap:ergodic}
Before providing the proof, we are first going to provide two lemmas:

\begin{lemma}\label{lem:avgconvex}
    For $f: \X\rightarrow \mathbb{R}$ $\rho$-convex, $(x^i)\in \X$, $\bar{x}^i:=\frac{1}{i}\sum_{j=1}^i x^j$, and all $N\in\mathbb{N}$,
    \begin{equation*}
        f(\bar{x}^N) \leqslant \bar{f} - \frac{{\rho}}{2N}\sum_{i=1}^N\frac{i-1}{i}\|x^i-\bar{x}^{i-1}\|_{\X}^2,
    \end{equation*}
    and furthermore emphasising the weak convex case, for any $x\in\X$:
    \begin{align*}
        f(\bar{x}^N)\leqslant \bar{f}&+\max{(0,-\rho)}\cdot %
        \left[\frac{1}{N}\sum_{i=1}^N\|x^{i}-x\|_{\X}^2 + \frac{1}{N}\sum_{i=1}^N\frac{1}{i}\sum^{i-1}_{j=1}\|x^j-x\|_{\X}^2 \right].%
    \end{align*}
\end{lemma}
\begin{proof}
This arises from the following fact for $\rho$-convex functions: 
    \begin{equation}\begin{split}
        f\left(\frac{1}{N}\left(x_1+\dots+x_N\right)\right) &\leqslant \frac{1}{N}f(x_N) - \frac{\rho}{2}\frac{1}{N}\left(1-\frac{1}{N}\right)\|x_n-\bar{x}_{n-1}\|_{\X}^2 \\& \qquad\qquad + \frac{N-1}{N}f\left(\frac{1}{N-1}\left(x_1+\dots+x_{N-1}\right)\right) \\&\leqslant\dots\leqslant \bar{f} - \frac{\rho}{2N}\sum_{i=1}^k\frac{i-1}{i}\|x_i-\bar{x}_{i-1}\|_{\X}^2.
        \label{eq:lemavgfirstresult}
        \end{split}
    \end{equation}
To arrive at the second result we wish to expand the last term in terms of iterate lengths:
\begin{align*}
\sum_{i=1}^N\frac{i-1}{i}\|x_i-\bar{x}_{i-1}\|_{\X}^2 &\leqslant 2\sum_{i=1}^N\frac{i-1}{i}\left(\|x_i-x\|_{\X}^2 + \|x-\bar{x}_{i-1}\|_{\X}^2\right)\\
&\leqslant 2\sum_{i=1}^N\frac{i-1}{i}\left(\|x_i-x\|_{\X}^2 + \left\|\frac{1}{i-1}\sum_{j=1}^{i-1} \left(x_{j}-x\right)\right\|_{\X}^2\right)\\
\text{By Jensen's inequality: } & \leqslant 2\sum_{i=1}^N\frac{i-1}{i}\left(\|x_i-x\|_{\X}^2 + \frac{1}{i-1}\sum^{i-1}_{j=1}\|x_j-x\|_{\X}^2 \right).
\end{align*}
Plugging this into \Cref{eq:lemavgfirstresult} yields the desired result.
\end{proof}
And now by application of \Cref{lem:avgconvex} to $L$:
\begin{lemma}\label{thm:ergodic}
Let $\bar{z}^k:=\left(\bar{x}^k, \bar{y}^k\right):=\frac{1}{k} \sum_{i=1}^k z^i$ be the average iterate. Assuming that $L(x,y)$ is $\rho$-convex in the first argument and $\mu$-concave in the second, for any $k \geqslant 1$ and z:
\begin{align*}
L\left(\bar{x}^k, y\right)-L\left(x, \bar{y}^k\right) &  \leqslant \frac{\left\|z-z^0\right\|_M^2}{2 k} - \frac{1}{2k}\sum_{i=1}^k\|z^{i-1}-z^{i}\|_M^2
    \\ &\quad -\frac{\wcv{\rho}}{2k}\sum_{i=1}^k\|x^{i}-x\|_{\X}^2 - \frac{\wcv{\rho}}{2k}\sum_{i=1}^k\frac{i-1}{i}\|x^i-\bar{x}^{i-1}\|_{\X}^2 
    \\ &\quad -\frac{\scv{\mu}}{2k}\sum_{i=1}^k\|y^{i}-y\|_{\Y}^2 - \frac{\scv{\mu}}{2k}\sum_{i=1}^k\frac{i-1}{i}\|y^{i}-\bar{y}^{i-1}\|_{\Y}^2.\\
\end{align*}
Further, setting $D=\operatorname{diag}\left(\max{(0,-\rho)},\max{(0,-\mu)}\right)$:
\begin{equation*}
        L\left(\bar{x}^k, y\right)-L\left(x, \bar{y}^k\right) \leqslant \frac{\left\|z-z^0\right\|_M^2}{2 k} + \frac{3}{2k}\sum_{i=1}^k\|z^{i}-z\|_D^2 + \frac{1}{k}\sum_{i=1}^k\frac{1}{i}\sum^{i-1}_{j=1}\|z^j-z\|_D^2.
\end{equation*}
\end{lemma}
\begin{proof}
We can expand the original as: 
\begin{equation*}
    L\left(\bar{x}^k, y\right)-L\left(x, \bar{y}^k\right) = \left(L\left(\bar{x}^k, y\right)-L\left(x, y\right)\right) + \left(L\left(x, y\right)-L\left(x, \bar{y}^k\right)\right)
\end{equation*}
Noting that the first bracket is convex in $\bar{x}^k$ and second is convex in $\bar{y}^k$, we can use \Cref{lem:avgconvex} to bound:
\begin{align*}
    L\left(\bar{x}^k, y\right)-L\left(x, \bar{y}^k\right)&=\left(L\left(\bar{x}^k, y\right)-L\left(x, y\right)\right) + \left(L\left(x, y\right)-L\left(x, \bar{y}^k\right)\right)\\
    &\leqslant\frac{1}{k} \sum_{i=1}^{k} L\left(x^{i}, y\right)-L\left(x, y^{i}\right) 
    \\ &\qquad -\frac{\wcv{\rho}}{2k}\sum_{i=1}^k\|x^{i}-x\|_{\X}^2 - \frac{\wcv{\rho}}{2k}\sum_{i=1}^k\frac{i-1}{i}\|x^i-\bar{x}^{i-1}\|_{\X}^2 
    \\ &\qquad -\frac{\scv{\mu}}{2k}\sum_{i=1}^k\|y^{i}-y\|_{\Y}^2 - \frac{\scv{\mu}}{2k}\sum_{i=1}^k\frac{i-1}{i}\|y^{i}-\bar{y}^{i-1}\|_{\Y}^2\\
    \text{by \Cref{eq:k_gap}}&\leqslant  \frac{\left\|z-z^0\right\|_M^2}{2 k} - \frac{1}{2k}\sum_{i=1}^k\|z^{i-1}-z^{i}\|_M^2
    \\ &\qquad -\frac{\wcv{\rho}}{2k}\sum_{i=1}^k\|x^{i}-x\|_{\X}^2 - \frac{\wcv{\rho}}{2k}\sum_{i=1}^k\frac{i-1}{i}\|x^i-\bar{x}^{i-1}\|_{\X}^2 
    \\ &\qquad -\frac{\scv{\mu}}{2k}\sum_{i=1}^k\|y^{i}-y\|_{\Y}^2 - \frac{\scv{\mu}}{2k}\sum_{i=1}^k\frac{i-1}{i}\|y^{i}-\bar{y}^{i-1}\|_{\Y}^2
\end{align*}
Using the second part of \Cref{lem:avgconvex}, we can arrive at the second result: 
\begin{align*}
    L\left(\bar{x}^k, y\right)-L\left(x, \bar{y}^k\right)&\leqslant  \frac{\left\|z-z^0\right\|_M^2}{2 k} -\frac{\wcv{\rho}}{2k}\sum_{i=1}^k\|x^{i}-x\|_{\X}^2 -\frac{\scv{\mu}}{2k}\sum_{i=1}^k\|y^{i}-y\|_{\Y}^2
     \\& \qquad+ \max{(0,-\rho)}\cdot\left[\frac{1}{k}\sum_{i=1}^k\|x^{i}-x\|_{\X}^2 + \frac{1}{k}\sum_{i=1}^k\frac{1}{i}\sum^{i-1}_{j=1}\|x^j-x\|_{\X}^2 \right]
     \\&\qquad+ \max{(0,-\mu)}\cdot\left[\frac{1}{k}\sum_{i=1}^k\|y^{i}-y\|_{\Y}^2 + \frac{1}{k}\sum_{i=1}^k\frac{1}{i}\sum^{i-1}_{j=1}\|y^j-y\|_{\Y}^2 \right]
\end{align*}
Yielding the desired result.
\end{proof}

Now, with the lemmas above we can prove \Cref{thm:ergodic_rate}.
\begin{proof}[Proof of \Cref{thm:ergodic_rate}]
Using $\scv{\mu}$-strong concavity in the second argument and $\wcv{\rho}$-weak convexity in the first, and using that $\hat{z}$ is a stationary point of $L$, we can bound:
\begin{equation*}
    L\left(\bar{x}^k, y\right)-L\left(x, \bar{y}^k\right) \leqslant -\frac{\mu}{2}\|y-\hat{y}\|_{\Y}^2+\frac{\rho}{2}\|x-\hat{x}\|_{\X}^2 + L\left(\bar{x}^k, \hat{y}\right)-L\left(\hat{x}, \bar{y}^k\right).
\end{equation*}
Now, from the second part of \Cref{thm:ergodic} we have 
\begin{equation*}
    L\left(\bar{x}^k, \hat{y}\right)-L\left(\hat{x}, \bar{y}^k\right) \leqslant \frac{\left\|\hat{z}-z^0\right\|_M^2}{2 k}  +\frac{3}{2k}\sum_{i=1}^k\|z^{i}-\hat{z}\|_D^2 + \frac{1}{k}\sum_{i=1}^k\frac{1}{i}\sum^{i-1}_{j=1}\|z^j-\hat{z}\|_D^2,
\end{equation*}
while from \Cref{thm:KLconvergence} (by redefining the constants appropriately), then we have constants $\nu>0,0<\tau<1$, such that
$$
\left\|z^k-\hat{z}\right\|_D \leqslant \nu \tau^{k},
$$
which in turn implies the following two bounds:
\begin{equation*}
\frac{1}{k}\sum_{i=1}^k\|z^{i}-\hat{z}\|_D^2 \leqslant \frac{1}{k}\nu^2\sum_{i=1}^k \tau^{2k} = \frac{1}{k}\nu^2\tau^2\frac{1-\tau^{2k}}{1-\tau^2} = \mathcal{O}\left(\frac{1}{k}\right)
\end{equation*}
and
\begin{align*}
\frac{1}{k}\sum_{i=1}^k\frac{1}{i}\sum^{i-1}_{j=1}\|z^j-\hat{z}\|_D^2 &\leqslant \frac{1}{k}\nu^2\tau^2 \sum_{i=1}^k\frac{1}{i}\cdot\frac{1-\tau^{2(i-1)}}{1-\tau^2} \\&\leqslant \frac{1}{k}\frac{\nu^2\tau^2}{1-\tau^2} \sum_{i=1}^k\frac{1}{i} \leqslant \frac{\nu^2\tau^2}{1-\tau^2}\frac{\log{k}+1}{k} = \mathcal{O}\left(\frac{\log{k}}{k}\right),
\end{align*}
by a standard inequality bounding harmonic series.

If instead $\theta \in\left(\frac{1}{2}, 1\right)$, then there exist a constant $\mu>0$,
$$
\left\|z^k-\hat{z}\right\|_D \leqslant \mu k^{-\frac{1-\theta}{2 \theta-1}}.
$$
But this implies the following two bounds:
\begin{equation*}
\frac{1}{k}\sum_{i=1}^k\|z^{i}-\hat{z}\|_D^2 \leqslant \frac{1}{k}\mu^2\sum_{i=1}^k i^{-2\frac{1-\theta}{2 \theta-1}} \leqslant \frac{1}{k}\mu^2\zeta\left(2\frac{1-\theta}{2 \theta-1}\right)=\mathcal{O}\left(\frac{1}{k}\right)
\end{equation*}
and
\begin{align*}
\frac{1}{k}\sum_{i=1}^k\frac{1}{i}\sum^{i-1}_{j=1}\|z^j-\hat{z}\|_D^2 &\leqslant \frac{1}{k} \sum_{i=1}^k\frac{1}{i}\mu^2\zeta\left(2\frac{1-\theta}{2 \theta-1}\right) \\&\leqslant \frac{1}{k}\mu^2\zeta\left(2\frac{1-\theta}{2 \theta-1}\right)(\log{k}+1) = \mathcal{O}\left(\frac{\log k}{k}\right),
\end{align*}
where $\zeta$ denotes the Riemann zeta function.

Now, if $\theta=0$, denoting $\sum_{i=1}^\infty\|z^{i}-\hat{z}\|_D^2$ as $d$, which is guarranteed to be finite, since we converge in a finite number of steps by \Cref{thm:KLconvergence}, we have:
\begin{equation*}
\frac{1}{k}\sum_{i=1}^k\|z^{i}-\hat{z}\|_D^2 \leqslant \frac{d}{k} = \mathcal{O}\left(\frac{1}{k}\right) \quad 
\end{equation*}
and
\begin{align*}
\frac{1}{k}\sum_{i=1}^k\frac{1}{i}\sum^{i-1}_{j=1}\|z^j-\hat{z}\|_D^2 &\leqslant \frac{1}{k} \sum_{i=1}^k\frac{1}{i}d \leqslant d\frac{\log{k}+1}{k} = \mathcal{O}\left(\frac{\log{k}}{k}\right).
\end{align*}
Therefore, in all three cases we achieve the following bound:
\begin{equation*}
L\left(\bar{x}^k, y\right)-L\left(x, \bar{y}^k\right) \leqslant \frac{1}{2}\|z-\hat{z}\|_D^2+\mathcal{O}\left(\frac{\log k}{k}\right).
\end{equation*}
\end{proof}

\section{Proof of \Cref{thm:universal_IWCNN}}\label{ap:universal_IWCNN}
Before we can move onto proof of \Cref{thm:universal_IWCNN}, we are going to need to establish a number of results relating to both universal approximation properties of standard neural networks, as well as approximation properties of Moreau envelopes. 
\begin{lemma}[Lemma 3 in \cite{10.1007/BF02559552}]\label{lem:moreauenv}
Let $f$ be a proper extended-real-valued l.s.c. bounded below function on $\X$. 

For any $t>0$ it holds that
\begin{itemize}
    \item[(i)] $f_t$ is a real-valued minorant of $f$;
    \item [(ii)] $\inf f_t=\inf f$;
    \item [(iii)] $\arg \min f_t=\arg \min f$.
\end{itemize}
Furthermore, when $t \downarrow 0$:
\begin{itemize}
    \item[(iv)] $f_t \rightarrow f$ pointwise;
    \item[(v)] $f_t \rightarrow f$ with respect to the epi-distance topology provided $\|\cdot\|_{\X}$ is uniformly rotund;
    \item[(vi)] $f_t \rightarrow f$ uniformly on bounded sets when $f$ is uniformly continuous on bounded sets and $\|\cdot\|_{\X}$ is uniformly rotund.
\end{itemize}
\end{lemma}
Note, that e.g. the Euclidean norm is uniformly rotund. In what follows we use point (vi) in order to approximate the desired function with a weak convex smooth approximant.
\begin{theorem}[{Representation power of ICNN; Theorem 1 in  \cite{chen2018optimal}}]
\label{thm:icnn_approx}
    For any Lipschitz convex function over a compact domain, there exists a neural network with nonnegative weights and ReLU activation functions (i.e., an ICNN) that approximates it uniformly within $\varepsilon$, i.e. in the $\|\cdot\|_\infty$ norm.
\end{theorem}
A further summary can be found in \cite{Hanin_2019}. The last piece for the proof of universal approximation of IWCNNs is a universal approximation result for neural networks with smooth activations:
\begin{theorem}[Theorem 3.2 in \cite{kidger2020universal}]\label{thm:nn_univapprox}
Let $\rho: \mathbb{R} \rightarrow \mathbb{R}$ be any nonaffine continuous function which is continuously differentiable at at least one point, with nonzero derivative at that point. Let $K \subseteq \mathbb{R}^n$ be compact. Then $\mathcal{N N}_{n, m, n+m+2}^\rho$ is dense in $C\left(K ; \mathbb{R}^m\right)$ with respect to the uniform norm.
\end{theorem}
Here $\mathcal{N N}_{n, m, n+m+2}^\rho$ represents the class of functions
$\mathbb{R}^n\to\mathbb{R}^m$ described by feedforward neural networks with $n$ neurons in the input layer, $m$ neurons
in the output layer, and an arbitrary number of hidden layers, each with $n+m+2$ neurons and activation function $\rho$ as in \cite{kidger2020universal}.

With these results, we can now prove \Cref{thm:universal_IWCNN}.
\begin{proof}[Proof of \Cref{thm:universal_IWCNN}]
Fix $\varepsilon>0$. Consider some target proper extended-real-valued l.s.c. function $f\in\mathcal{C}_0(X)$, with $X$ compact. Then, by %
\cite{sun2019least} Theorem 6, there exists $f_{wc}\in \mathcal{W C}(X) \cap \mathcal{C}_0(X)$ with $\|f-f_{wc}\|_{\infty}<\frac{\varepsilon}{2}$. In fact, we know exactly how to find such approximation using the Moreau envelope: $f_{wc} = f_t$ for some $t>0$ by \Cref{lem:moreauenv}, which furthermore is Lipschitz smooth by \Cref{thm:moreauenv_wc}. As a result, we can write $f_{wc} = \operatorname{Id}\circ\;f_t$. Now, more generically for $m\neq1$, we can e.g. consider just the first element, i.e. $g^\textrm{conv}: x \to x_1$ and $g^\textrm{sm}: x\to(e_{t}f,0,\cdots,0)^\top$, with $f_{wc}= g^\textrm{conv}\circ g^\textrm{sm}$.

By \Cref{thm:nn_univapprox} there exists a neural network such that $\|g^\textrm{sm}_{\theta_2}-g^\textrm{sm}\|_{\infty}<\frac{\varepsilon}{2}$. Since $g^\textrm{sm}_{\theta_2}$ is continuous and $X$ is compact, $g^\textrm{sm}_{\theta_2}(X)$ is compact. Therefore, by \Cref{thm:icnn_approx}, as $g^\textrm{conv}$ is Lipschitz convex, there exists an ICNN such that $\|g^\textrm{ICNN}_{\theta_1}-g^\textrm{conv}\|_{\infty}<\frac{\varepsilon}{2}$ on $\operatorname{Im} g^\textrm{sm}_{\theta_2}$\footnote{Technically in this we are approximating the identity map, which can be exact with \texttt{ReLU}/\texttt{Leaky ReLU} activations.}.
Therefore, 
\begin{align*}
\|g^\textrm{ICNN}_{\theta_1} \circ g^\textrm{sm}_{\theta_2} - g^\textrm{conv}\circ g^\textrm{sm}\|_{\infty}      
&\leqslant \|g^\textrm{ICNN}_{\theta_1} \circ g^\textrm{sm}_{\theta_2} - g^\textrm{conv} \circ g^\textrm{sm}_{\theta_2}\|_{\infty}+\|g^\textrm{conv} \circ g^\textrm{sm}_{\theta_2} - g^\textrm{conv}\circ g^\textrm{sm}\|_{\infty}\\
&\leqslant \|g^\textrm{ICNN}_{\theta_1} - g^\textrm{conv}\|_{\infty}+\| g^\textrm{sm}_{\theta_2} - g^\textrm{sm}\|_{\infty} \\&< \varepsilon,
\end{align*}
by 1-Lipschitzness of $g^\textrm{conv}$. Therefore, IWCNNs are dense in the space of continuous functions. 
\end{proof}

\section{Experimental set-up and additional data visualisations for \Cref{sec:results}}
\subsection{Experimental set-up} 
\label{ap:expsetup}

We use human abdominal CT scans for 10 patients provided by Mayo Clinic for the low-dose CT grand challenge \cite{mayo_ct_challenge}. The training dataset for CT experiments consists of a total of 2250 2D slices, each of dimension $512 \times 512$, corresponding to 9 patients. The remaining 128 slices corresponding to one patient are used to evaluate the reconstruction performance. 

Projection data is simulated in ODL \cite{odl} with a GPU-accelerated \textit{astra} back-end, using a parallel-beam acquisition geometry with 350 angles and 700 rays/angle, using additive Gaussian noise with $\sigma=3.2$. The pseudoinverse reconstruction is taken to be images obtained using FBP. For limited angle experiments, data is simulated with a missing angular wedge of 60$^{\circ}$. The native \texttt{odl} power method is used to approximate the norm of the operator.

The TV method was computed by employing the ADMM-based solver in ODL. 

All of the methods are implemented in \textit{PyTorch}~\cite{paszke2017automatic}. The LPD method is trained on pairs of target images and projected data, while the U-Net post-processor is trained on pairs of true images and the corresponding reconstructed FBPs. AR, ACR, ACNCR, and AWCR, in contrast, require ground-truth and FBP images drawn from their marginal distributions (and hence not necessarily paired). The hyperparameter $\mu_0$ from \Cref{def:AWCR} is chosen to be the same as in \cite{ACR}: $\mu_0 = \log\left(1+\exp(-9.0)\right)$, and during training diverges from this value minimally, having almost no effect on the convexity constant of the regulariser. To aid with training, $\lambda$ from \Cref{eq:AR_loss} is first chosen to be small (0.1) and once the network is trained is increased to a larger value (10) as in \cite{ACR,AR}), and the network is fine-tuned. The \textit{RMSprop} optimizer, following the recommendation in \cite{AR,ACR} with a learning rate of $5\times 10^{-5}$ is used for training for a total of 50 epochs. For fine-tuning, the learning rate is reduced to $10^{-6}$. The AWCR architecture differed for the two experiments in the following way:

\textbf{Sparse view CT}\quad 
The ICNN component of the AWCR is constructed with 5 convolutional layers, using \texttt{LeakyReLU} activations and $5\times 5$ kernels with 16 channels. The smooth component of the AWCR is constructed using 6 convolutional layers, using \texttt{nn.SiLU} activations and $5\times 5$ kernels with a doubling number of channels from 16 and stride of 2, similar to that of \cite{AR}, with the last layer containing 128 channels

\textbf{Limited view CT}\quad 
The ICNN component of the AWCR is constructed with 5 convolutional layers, using \texttt{LeakyReLU} activations and $5\times 5$ kernels with 16 channels. Guided by the observations of \cite{ACR}, that reducing the total number of layers and the number of feature channels helps avoid overfitting in the limited angle setting, the smooth component of the AWCR is constructed using a single convolutional layers, using a \texttt{SiLU} activation, $7\times 7$ with 32 channels. To remain close to the ACR formulation, a residual structure is further employed.

The reconstruction in all adversarial regularisation cases is performed by solving the variational problem via gradient-descent for 1000 iterations with a step size of $10^{-6}$. For the AWCR-PD, the primal-dual algorithm is used for solving the variational problem. The proximal operator of the network is approximated by performing gradient descent with backtracking on the objective in \Cref{eq:pdhgm_update}. The step sizes are chosen to be equal to $0.1/\|\A\|$ in the sparse case, and in the limited angle are chosen to be equal to $10/\|\A\|$ and $0.1/\|\A\|$ to ensure that iterates do not diverge and that steps are large enough to provide good results. Akin to AR, reconstruction performance of AWCR can sometimes deteriorate if early stopping is not applied. For a fair comparison, we report the highest PSNR achieved by all methods during reconstruction. The regularisation parameter $\alpha$ is chosen according to \cite{AR} and is \emph{not} tuned.

\subsection{Additional data visualisations }
\label{ap:datavis}

\begin{figure*}[htp]
  \centering
  \begin{minipage}[t]{0.2\linewidth}
  \centering
  \vspace{0pt}
  \begin{tikzpicture}[spy using outlines={circle,red,magnification=3.0,size=1.2cm, connect spies}]   
    \node {\includegraphics[width=\linewidth]{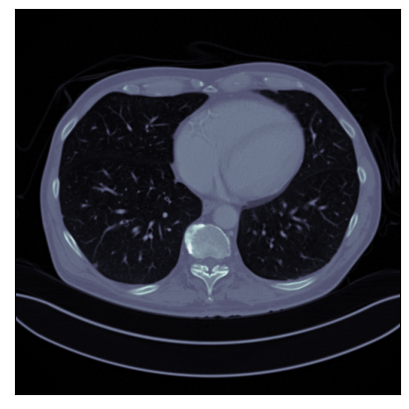}};
    \spy on (0.045,-0.6) in node [left] at (1.7,1.25);
    \spy on (-1.14,-1.06) in node [left] at (-0.4,1.25);  
  \end{tikzpicture}
  \vskip-0.5\baselineskip
  {\scriptsize Ground-truth}
  \end{minipage}\hfill
  \begin{minipage}[t]{0.2\linewidth}%
  \centering
  \vspace{0pt}
  \begin{tikzpicture}[spy using outlines={circle,red,magnification=3.0,size=1.2cm, connect spies}]   
    \node {\includegraphics[width=\linewidth]{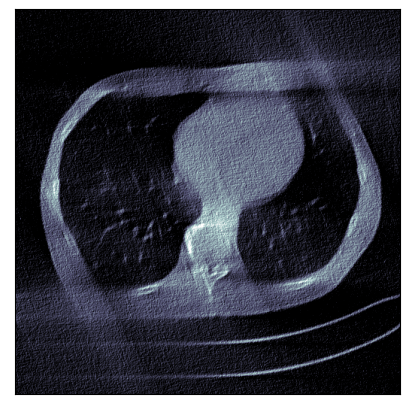}};
    \spy on (0.045,-0.6) in node [left] at (1.7,1.25);
    \spy on (-1.14,-1.06) in node [left] at (-0.4,1.25);  
  \end{tikzpicture}
  \vskip-0.5\baselineskip
  {\scriptsize FBP: 18.232 dB, 0.223}
  \end{minipage}\hfill
  \begin{minipage}[t]{0.2\linewidth}%
  \centering
  \vspace{0pt}
  \begin{tikzpicture}[spy using outlines={circle,red,magnification=3.0,size=1.2cm, connect spies}]   
    \node {\includegraphics[width=\linewidth]{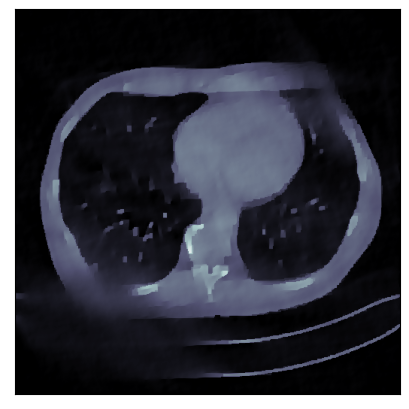}};
    \spy on (0.045,-0.6) in node [left] at (1.7,1.25);
    \spy on (-1.14,-1.06) in node [left] at (-0.4,1.25);  
  \end{tikzpicture}
  \vskip-0.5\baselineskip
  {\scriptsize TV: 23.694 dB, 0.749}
  \end{minipage}\hfill
  \begin{minipage}[t]{0.2\linewidth}%
  \centering
  \vspace{0pt}
  \begin{tikzpicture}[spy using outlines={circle,red,magnification=3.0,size=1.2cm, connect spies}]   
    \node {\includegraphics[width=\linewidth]{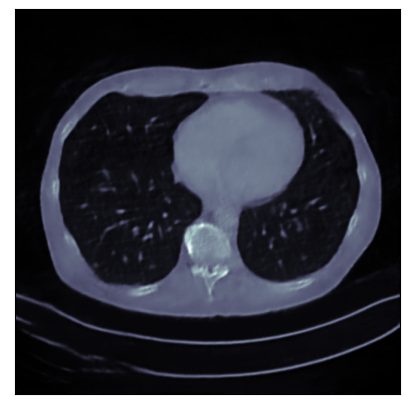}};
    \spy on (0.045,-0.6) in node [left] at (1.7,1.25);
    \spy on (-1.14,-1.06) in node [left] at (-0.4,1.25);  
  \end{tikzpicture}
  \vskip-0.5\baselineskip 
  {\scriptsize U-Net: 27.943 dB, 0.825}
  \end{minipage}\hfill
  \begin{minipage}[t]{0.2\linewidth}%
  \centering
  \vspace{0pt}
  \begin{tikzpicture}[spy using outlines={circle,red,magnification=3.0,size=1.2cm, connect spies}]   
    \node {\includegraphics[width=\linewidth]{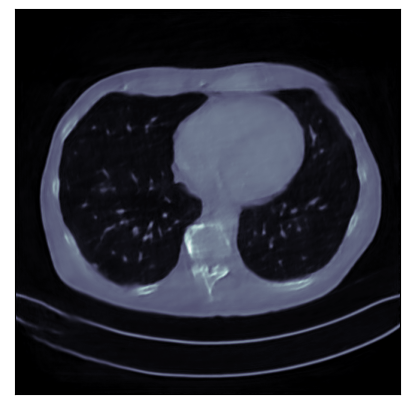}};
    \spy on (0.045,-0.6) in node [left] at (1.7,1.25);
    \spy on (-1.14,-1.06) in node [left] at (-0.4,1.25);  
  \end{tikzpicture}
  \vskip-0.5\baselineskip
  {\scriptsize LPD: 28.150 dB, 0.821}
  \end{minipage}\hfill
  \begin{minipage}[t]{0.2\linewidth}%
  \centering
  \vspace{0pt}
  \begin{tikzpicture}[spy using outlines={circle,red,magnification=3.0,size=1.2cm, connect spies}]   
    \node {\includegraphics[width=\linewidth]{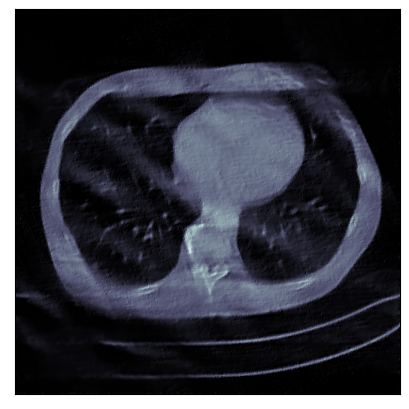}};
    \spy on (0.045,-0.6) in node [left] at (1.7,1.25);
    \spy on (-1.14,-1.06) in node [left] at (-0.4,1.25);  
  \end{tikzpicture}
  \vskip-0.5\baselineskip
  {\scriptsize AR: 24.959 dB, 0.648}
  \end{minipage}\hfill
  \begin{minipage}[t]{0.2\linewidth}%
  \centering
  \vspace{0pt}
  \begin{tikzpicture}[spy using outlines={circle,red,magnification=3.0,size=1.2cm, connect spies}]   
    \node {\includegraphics[width=\linewidth]{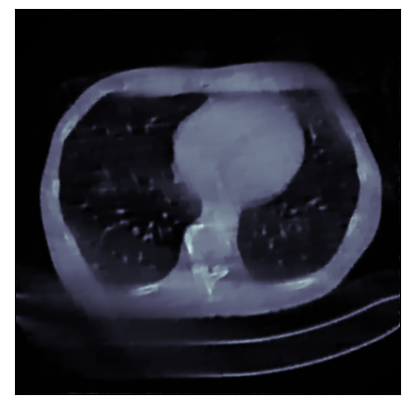}};
    \spy on (0.045,-0.6) in node [left] at (1.7,1.25);
    \spy on (-1.14,-1.06) in node [left] at (-0.4,1.25);  
  \end{tikzpicture}
  \vskip-0.5\baselineskip
  {\scriptsize ACR: 24.278 dB, 0.779}
  \end{minipage}\hfill
   \begin{minipage}[t]{0.2\linewidth}%
  \centering
  \vspace{0pt}
  \begin{tikzpicture}[spy using outlines={circle,red,magnification=3.0,size=1.2cm, connect spies}]   
    \node {\includegraphics[width=\linewidth]{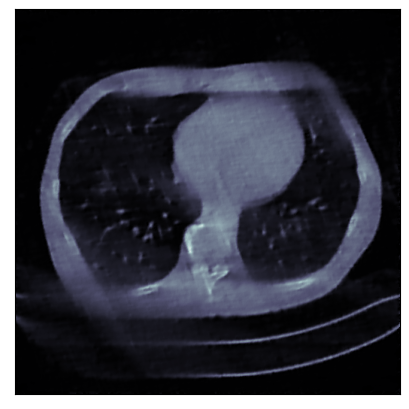}};
    \spy on (0.045,-0.6) in node [left] at (1.7,1.25);
    \spy on (-1.14,-1.06) in node [left] at (-0.4,1.25);  
  \end{tikzpicture}
  \vskip-0.5\baselineskip
  {\scriptsize ACNCR: 26.117 dB, 0.796}
  \end{minipage}\hfill
   \begin{minipage}[t]{0.2\linewidth}%
  \centering
  \vspace{0pt}
  \begin{tikzpicture}[spy using outlines={circle,red,magnification=3.0,size=1.2cm, connect spies}]   
    \node {\includegraphics[width=\linewidth]{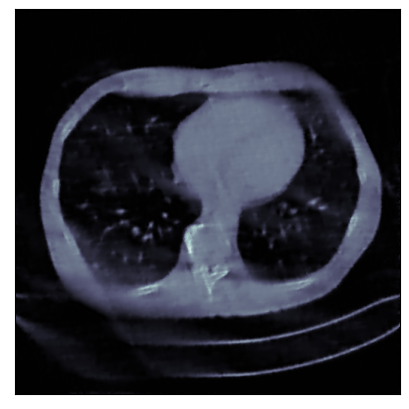}};
    \spy on (0.045,-0.6) in node [left] at (1.7,1.25);
    \spy on (-1.14,-1.06) in node [left] at (-0.4,1.25);  
  \end{tikzpicture}
  \vskip-0.5\baselineskip
  {\scriptsize AWCR: 26.213 dB, 0.766}
  \end{minipage}\hfill
   \begin{minipage}[t]{0.2\linewidth}%
  \centering
  \vspace{0pt}
  \begin{tikzpicture}[spy using outlines={circle,red,magnification=3.0,size=1.2cm, connect spies}]   
    \node {\includegraphics[width=\linewidth]{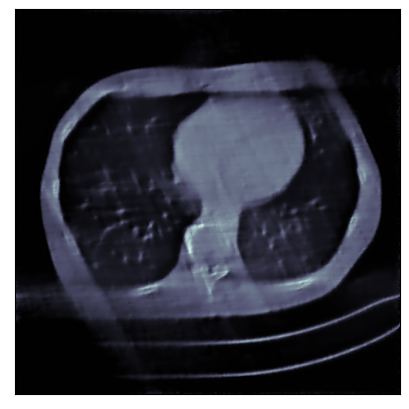}};
    \spy on (0.045,-0.6) in node [left] at (1.7,1.25);
    \spy on (-1.14,-1.06) in node [left] at (-0.4,1.25);  
  \end{tikzpicture}
  \vskip-0.5\baselineskip
  {\scriptsize AWCR-PD: 24.593 dB, 0.751}
  \end{minipage}\hfill
  \caption{\small{Reconstructed images obtained using different methods, along with the associated PSNR and SSIM, for limited view CT.  } \label{fig:visual2}}
\end{figure*}

\end{document}